\documentclass[12pt]{article}
\newcommand{\hf}{\textstyle{\frac{1}{2}}}

\usepackage{natbib,amsmath,amssymb,amsthm,amsfonts}
\usepackage{graphicx,url}
\usepackage{bm}
\usepackage{latexsym}
\usepackage{color}
\usepackage{hyperref}
\usepackage{bbm}
\usepackage{authblk}

\newcommand{\E}{\mathbb{E}}
\newcommand{\mP}{\mathbb{P}}
\newcommand{\mN}{\mathbb{N}}
\newcommand{\mR}{\mathbb{R}}
\newcommand{\mC}{\mathbb{C}}

\newcommand{\cR}{\mathcal{R}}
\newcommand{\cRn}{\mathcal{R}_\nu}

\DeclareMathOperator{\tr}{tr}
\DeclareMathOperator{\plim}{plim}
\DeclareMathOperator{\aslim}{a.s.lim}
\DeclareMathOperator{\diag}{diag}

\newcommand{\lhn}{\hat{\ell}_\nu}
\newcommand{\lhat}{\hat{\ell}}
\newcommand{\lmin}{\lambda_{\rm \min}}
\newcommand{\rhonn}{\rho_{\nu n}}

\newcommand{\toas}{\stackrel{\rm a.s.}{\to}}
\newcommand{\tooas}{\stackrel{\rm a.s.}{\longrightarrow}}
\newcommand{\toodist}{\stackrel{\mathcal{D}}{\longrightarrow}}

\newcommand{\ssf}{\mathsf{F}}
\newcommand{\ssm}{\mathsf{m}}
\newcommand{\hprod}{\circ}

\newcommand{\bx}{\bm{x}}
\newcommand{\by}{\bm{y}}

\theoremstyle{plain}
\newtheorem{theorem}{Theorem}
\newtheorem{proposition}[theorem]{Proposition}

\newtheorem{lemma}[theorem]{Lemma}

\begin{document}\title{\Large \textbf{Draft Report: TITLE?}}

\title{Notes on asymptotics of sample eigenstructure for 
  spiked covariance models with non-Gaussian data}
\author{Iain M. Johnstone, Jeha Yang}
\affil{Stanford University}
\date{\today}

\vspace{-1in}

\maketitle

\begin{abstract}
  These expository notes serve as a reference for an accompanying post
  \citet{mjmy18}.
  In the spiked covariance model, we develop results on asymptotic
  normality  of sample leading eigenvalues and certain projections of
  the corresponding sample eigenvectors.
  The results parallel those of \citet{paul07}, but are given using the
  non-Gaussian model of \citet{bayo08}.
  The results are not new, and citations are given, but proofs are
  collected and organized as a point of departure for \citet{mjmy18}.
\end{abstract}

\section{Introduction}
\label{sec:introduction}

Principal Components Analysis (PCA) is typically introduced as a spectral
decomposition of sample covariance matrices.
In practice, it is often applied after standardization of the data,
that is, after each variable is rescaled to have zero sample mean and unit
sample standard deviation.
In other words, PCA is performed on the sample correlation matrix.
It is therefore natural to ask how the standardization process affects
properties of PCA, for example consistency and asymptotic normality of
the sample eigenvalues and eigenvectors.
\citet{mjmy18} addresses this question in the high dimensional setting
in which the number of variables grows in proportion to sample size,
but the underlying correlation structure is assumed to be a low rank
perturbation of the identity. 

Their results are developed as analogs of those for sample covariance
matrices originally given by \citet{paul07}.
Since then, 
a substantial literature has contributed more general
models and different technical methods.
A particular selection of these models and methods has been
adopted and 
adapted for the sample correlation analysis of
\citet{mjmy18}.
These notes collect,
for comparison purposes, the corresponding tools in the
simpler sample covariance case.

\bigskip

We begin with the version of the spiked model for non-Gaussian data 
used by \cite{bayo08}, in the special case of simple spikes.

\textbf{Model M.} \ 
Let $x_1, x_2, \ldots$ be a sequence of random vectors independently
and identically distributed as $x \in \mR^{m+p}$
and such that $\E \|x\|^4 <\infty$.
Let the $(m+p)\times n$ data matrix $X = [x_1, \ldots, x_n]$.
We partition
\begin{equation*}
  X =
  \begin{bmatrix}
    X_1 \\ X_2
  \end{bmatrix}, \qquad \quad 
  x =
  \begin{bmatrix}
    \xi \\ \eta
  \end{bmatrix}.
\end{equation*}
We assume that $\xi \in \mR^m$ has mean $0$ and covariance $\Sigma$
and is independent of $\eta \in \mR^p$, which has components 
that are i.i.d with mean $0$ and variance $1$. 
We suppose that $m$ is fixed, while $p$ and $n$ grow with
$\gamma_n = p/n \to \gamma > 0$.

Write the spectral decomposition of $\Sigma$ as
\begin{equation*}
  \Sigma = P \Lambda P^T, \qquad
  P = [p_{1}, \ldots, p_{m}], \qquad
  \Lambda = \diag(\ell_1, \ldots, \ell_m).
\end{equation*}
We suppose that eigenvalues $\ell_\nu$ of the population covariance matrix
$\Sigma$ satisfy
$$
\ell_1 > \cdots > \ell_{m_0} > 1 + \sqrt \gamma \geq \ell_{m_0 +1} \geq \cdots \ell_m,
$$
%


\bigskip
Let $S = n^{-1} X X^T$ be the sample covariance matrix, with
corresponding $\nu$th sample eigenvalue and eigenvector
\begin{equation*}
  S \mathfrak{u}_\nu = \lhat_\nu \mathfrak{u}_\nu.
\end{equation*}
Here $\hat{\ell}_1 \geq \hat{\ell}_2 \geq \cdots$ and we restrict
attention to indices $\nu = 1, \ldots, m$. 
Partition $S$ and $\mathfrak{u}_\nu$ as
\begin{equation}
   \label{eq:part}
  S =
  \begin{bmatrix}
    S_{11} & S_{12} \\ S_{21} & S_{22}
  \end{bmatrix}
    = n^{-1}
    \begin{bmatrix}
      X_1X_1^T & X_1X_2^T \\ X_2X_1^T & X_2X_2^T
    \end{bmatrix}, 
  \qquad 
  \mathfrak{u}_\nu =
  \begin{bmatrix}
    u_\nu \\ v_\nu
  \end{bmatrix}
\end{equation}

We recall some properties of the eigenvalues 
$\mu_1 \geq \cdots \geq \mu_p$ of the $p\times p$ `noise' covariance
matrix $S_{22}$. Their empirical spectral
distribution (ESD), denoted $F_p$, converges weakly to the
Marchenko-Pastur (MP) law $F_\gamma$ with probability one. 
\citep{siba95}. The law $F_\gamma$ is 
supported on $[a_\gamma,b_\gamma] = [(1-\sqrt \gamma)^2,(1+\sqrt
\gamma)^2]$ if $\gamma \leq 1$ and 
on $\{0\} \cup [a_\gamma,b_\gamma]$ when $\gamma > 1$. 
We will also use the $n \times n$ companion matrix $C_n = n^{-1} X_2^T
X_2$, whose nonzero eigenvalues coincide with those of $n^{-1} X_2
X_2^T$. The ESD of $C_n$, denoted ${\sf{F}}_n$, converges weakly
a.s. to the 
`companion MP law'
\begin{equation}
  \label{eq:companion}
   \ssf_\gamma = (1-\gamma)I_{[0,\infty)} + \gamma F_\gamma.  
\end{equation}

\citet{bayi93} established convergence of the extreme non-trivial eigenvalues:
\begin{equation}
      \label{eq:bai-yin}
  \mu_1 \stackrel{\rm a.s.}{\longrightarrow}  b_\gamma \quad \text{and} \quad
  \mu_{p \wedge n} \stackrel{\rm a.s.}{\longrightarrow} a_\gamma. 
\end{equation}

For $\ell > 1 + \sqrt \gamma$, the `phase transition', define
\begin{align*}
  \rho(\ell,\gamma) & = \ell + \gamma \frac{\ell}{\ell-1}, \\
  \dot \rho(\ell, \gamma)
  = \frac{\partial}{\partial \ell} \rho(\ell, \gamma)
  & = 1 - \frac{\gamma}{(\ell -1)^2}.
\end{align*}
We use the abbreviations
\begin{equation} \label{eq:rhodef}
  \rho_\nu = \rho(\ell_\nu,\gamma), \qquad \quad
  \rhonn = \rho(\ell_\nu, \gamma_n), \qquad \quad
  \dot \rho_\nu = \dot \rho(\ell_\nu, \gamma).
\end{equation}
\citet{basi06} showed that
\begin{equation}
  \label{eq:Baik-S}
  \lhat_\nu \stackrel{\rm a.s.}{\longrightarrow}
  \begin{cases}
    \rho_\nu > b_\gamma & \text{if } \quad \ell_\nu > 1 + \sqrt \gamma \\
    b_\gamma            & \text{if } \quad \ell_\nu \leq 1 + \sqrt \gamma.
  \end{cases}
\end{equation}
[There is a similar `outlier' phenomenon for $\ell_\nu < 1 - \sqrt
\gamma$ (when $\gamma < 1$) associated with the smallest sample
eigenvalues.  We are interested here in outliers
above the bulk, so we do not consider this case, and focus only on the
$m$ largest eigenvalues.]



\textit{Notation for cumulants.} \ The fourth order cumulant
tensor of $\xi$ can be defined through partial derivatives $\partial_j
= \partial/\partial t_j$ at $t = 0$ of the cumulant generating function
\begin{equation} \label{eq:k-def}
  \kappa_{jklm} = \partial_j \partial_k \partial_l \partial_m
                  \log \E [\exp (t^T\xi)] \vert_{t=0}.
\end{equation}
When $\xi$ has mean zero, this may be written in terms of moments:
\begin{equation*}
  \kappa_{jklm} = \E [\xi_j \xi_k \xi_l \xi_m]
     - \sigma_{jk}\sigma_{lm}-\sigma_{jl}\sigma_{km}-\sigma_{jm}\sigma_{kl},
\end{equation*}
where $\Sigma = (\sigma_{jk})$.
From \eqref{eq:k-def}, one sees that when $\xi$ is Gaussian, then
$\kappa_{jklm} \equiv 0$, while if $\xi$ has independent
components, then $\kappa_{jklm}$ vanishes unless $j=k=l=m$.


Now some notation for contractions of the cumulant tensor.
Let
\begin{equation}\label{eq:kapmumu}
  K^\nu_{jj'}
   := \sum_{k,k'} p_{\nu,k} p_{\nu, k'} \kappa_{jj'kk'},
\end{equation}
Let $\mathcal{P}^{\mu \mu' \nu \nu'}$ denote the array with components
$ p_{\mu,i} p_{\mu ',j} p_{\nu,i'} p_{\nu ',j'}$, 
and write
$\mathcal{P}^{\nu}$ for $\mathcal{P}^{\nu \nu \nu \nu}$.
For two arrays $\mathcal{P}$ and $A$ 
with the same indexing, set
\begin{align}
\label{eq:innerproduct}
[\mathcal{P},A] := \sum_{i,j,i',j'} \mathcal{P}_{iji'j'} A_{iji'j'}.
\end{align}


\begin{theorem}
  \label{th:eval-normality}
Assume Model M, and suppose that $\ell_\nu > 1 + \sqrt \gamma$  is a
simple eigenvalue. As $p/n \to \gamma > 0$, we have
\begin{equation*}
  \sqrt n(\lhat_\nu - \rhonn) \toodist N(0,\sigma_\nu^2),
\end{equation*}
where $\rhonn$ is given by (\ref{eq:rhodef}) and
\begin{equation}
  \label{eq:sig-def}
  \sigma_\nu^2 = \sigma^2(\ell_\nu,\gamma) 
    = 2 \ell_\nu^2 \dot \rho_\nu + \dot \rho_\nu^2 \, [\mathcal{P}^\nu, \kappa] .
\end{equation}
\end{theorem}

Let $\mathfrak{p}_\nu$ denote the $\nu$th population eigenvector. From
 independence assumed in Model M,
\begin{equation*}
  \text{Cov}(x) = \text{Cov}   \begin{bmatrix}
    \xi \\ \eta
  \end{bmatrix}
  =
  \begin{pmatrix}
    \Sigma & 0 \\
    0 & I_p
  \end{pmatrix}, \qquad
  \mathfrak{p}_\nu =
  \begin{bmatrix}
    p_\nu \\ 0
  \end{bmatrix}.
\end{equation*}
Recalling \eqref{eq:part}, let $a_\nu = u_\nu/ \|u_\nu \|$ denote the
normalized subvector of $\mathfrak{u}_\nu$ restricted to the first $m$
co-ordinates. 
The vector $P^T a_\nu = ( p_1^T a_\nu, \ldots, p_{m}^T a_\nu)$ contains
the (signed magnitudes of) the projections of $a_\nu$ on each of the
population eigenvectors.
Being vectors of fixed dimension $m$, both $a_\nu$ and $P^T a_\nu$ are
consistent:
\begin{equation*}
  a_\nu \toas p_\nu, \qquad \quad
  P^T a_\nu \toas e_\nu,
\end{equation*}
where $e_\nu$ is the $\nu$th co-ordinate basis vector in $\mR^n$.

However, in the high dimensional setting, consistency of the
subvectors does not imply that of the full sample eigenvector
$\mathfrak{u_\nu}$:

\begin{theorem}
\label{th:cosines}
The squared inner product between $\nu$th sample and population eigenvector
\begin{equation}
  \label{eq:cos2}
  \langle \mathfrak{u}_\nu, \mathfrak{p}_\nu \rangle^2
   = \langle u_\nu, p_\nu \rangle^2
     \tooas
     \begin{cases}
\dfrac{\ell_\nu \dot \rho_\nu}{\rho_\nu}
     = \dfrac{1-\gamma/(\ell_\nu-1)^2}{1+\gamma/(\ell_\nu-1)}
       & \quad \text{if } \quad \ell_\nu > 1 + \sqrt \gamma \\
     0           & \quad \text{if } \quad \ell_\nu \leq 1 + \sqrt \gamma.
     \end{cases}
\end{equation}
\end{theorem}

The asymptotic fluctuation of the normalized subvectors $a_\nu =
u_\nu/ \|u_\nu\|$ is determined in part by the separation of
$\ell_\nu$ from the other spikes $\ell_\mu$. Define the diagonal
matrix
\begin{equation}
  \mathcal{D}_\nu
     =  \sum_{\mu \neq \nu}^{m} (\ell_{\nu}-\ell_{\mu})^{-1}
  e_{\mu} e_{\mu}^T.  \label{eq:Dnudef}
\end{equation}

\begin{theorem}
  \label{th:evec-normality}
Assume Model M, and suppose that $\ell_\nu > 1 + \sqrt \gamma$  is a
simple eigenvalue.  
As $p/n \to \gamma$, we have 
\begin{align}
  \sqrt n(P^T a_\nu- e_\nu)
     & \toodist N(0, \mathcal{D}_\nu \tilde{\Sigma}_\nu
       \mathcal{D}_\nu)  \notag \\ 
  \sqrt n( a_\nu- p_\nu)
     & \toodist N(0, P \mathcal{D}_\nu \tilde{\Sigma}_\nu
       \mathcal{D}_\nu P^T),  \notag 
       \intertext{where}
(\tilde \Sigma_\nu)_{\mu \mu'} 
  &=   \dot \rho_\nu^{-1} \ell_{\nu} \ell_{\mu}  \,
    \delta_{\mu, \mu'} +  [\mathcal{P}^{\mu  \mu'  \nu \nu}, \kappa].
    \label{eq:Sigma-til}
\end{align}
Here $\delta_{\mu, \mu'} = 1$ if $\mu = \mu'$ and $0$ otherwise
Note that the second term of \eqref{eq:Sigma-til} vanishes if either $\xi$ is
Gaussian or has independent components. 
\end{theorem}

Let us compare these statements with Paul's results, 
which imposed the extra assumptions that $\xi$ and
$\eta$ be Gaussian, with the components of $\xi$ independent, 
 $\gamma_n - \gamma = o(n^{-1/2})$ and $\gamma <
 1$.
The kurtosis terms drop out, $P = I_m$ and the asymptotic covariance
agrees with Paul:
\begin{equation*}
  \mathcal{D}_\nu \tilde{\Sigma}_\nu \mathcal{D}_\nu
  = \frac{1}{\dot \rho_\nu} \sum_{\mu \neq \nu}
        \frac{\ell_\nu \ell_\mu}{(\ell_\nu-\ell_\mu)^2} e_\mu e_\mu^T.
\end{equation*}
 

Centering at $\rhonn = \rho(\ell_\nu,\gamma_n)$ rather than at
$\rho_\nu = \rho(\ell_\nu,\gamma)$ is important:
if, for example, $\gamma_n = \gamma + an^{-1/2}$, then it is easily
seen that there is a limiting mean shift:
\begin{equation*}
  \sqrt n(\lhat_\nu - \rho_\nu) \toodist N(
  a\ell_\nu(\ell_\nu-1)^{-1}, \sigma_\nu^2).
\end{equation*}
On the other hand, it follows from Slutsky's theorem that there is no
harm, and indeed some benefit, see \citet{yajo17}, in replacing $\sigma_\nu$ with 
$\sigma_{\nu n} = \sigma(\ell_\nu,\gamma_n)$, so that
\begin{equation*}
  \sqrt n(\lhat_\nu - \rhonn)/\sigma_{\nu n} \toodist N(0,1).
\end{equation*}

\medskip
\textit{Relevant literature.} \ 
As noted, the results presented here have appeared previously: in
order to give citations, we describe some of the relevant literature
subsequent to \citet{paul07},
while not aspiring to a full review of work on the spiked model.
\citet{bayo08} consider Model M, which replaces the Gaussian
assumption with 4th moment assumptions on the data, allows a
non-diagonal covariance $\Sigma$, and permits spikes
with multiplicity greater than one both above $1 + \sqrt \gamma$ and
below $1 - \sqrt \gamma$ for $\gamma \in (0,1)$. 
They establish (a more general form of)
Theorem \ref{th:eval-normality}, centered at $\rho_\nu$. 
To do so, \citet{bayo08} develop a central
limit theorem for random 
sesquilinear forms $X^T B_n Y$ via the method of moments.
A martingale approach to this CLT, due to J. Baik and J. Silverstein,
is presented in \citet[Theorem 5.2 \& Appendix]{cdf09}, and is adopted
here, Appendix \ref{sec:appendix-b:-bai}.

Also working within Model M, \citet{lzw10} establish Theorem
\ref{th:cosines} (with convergence in probability), among other results.
\citet[Theorem 1.2]{shi13} establishes Theorems
\ref{th:eval-normality} and \ref{th:evec-normality} in Model M, along
with asymptotic joint distributions of the extreme eigenvalues and
vectors. Shi's Theorem 4.1 extends the moment method approach to the
Bai-Yao CLT, with an additional stochastic bound condition on the
entries of $B_n$. 

A generalized spiked model, in which the $\eta$ of Model M is extended
to $V_p^{1/2} \eta$ for deterministic matrices $V_p$ with a limiting
spectral distribution $H$, is discussed in \citet{bayo08} and studied
in detail in \citet{bayo12}, \citet{badi12} and \citet[Ch. 11]{yzb15}.
\citet{ding15} establishes a more general form of
Theorem \ref{th:evec-normality} in the generalized spike model,
building on the Bai-Yao CLT and \citet{badi12}.

Other papers study the asymptotic spectral structure of spiked
covariance models using different perturbation models than Model M.
For example \citet{bgrn11} and \citet{coha13} consider data
$(I_N + P)^{1/2} X$ with $P$ finite rank and $X$ unitarily invariant.
See also \citet{bggm11}, \citet{bkyy16}.


\medskip
\textit{Outline of paper.} \ 
The Schur complement of $S_{22} - t I_p$ is a matrix $K(t) - t
I_m$ of fixed dimension whose determinant vanishes a.s. at the sample
spike eigenvalues $\lhat_\nu$, and whose null space is one-dimensional
dimensional and lies along the (normalized) sample spike eigenvector
$a_\nu$.

We write $K(t)$ as a quadratic form $n^{-1} X_1 B_n(t) X_1^T$, where
$B_n(t)$ depends only on $X_2$ and so is independent of $X_1$.
A key role is played by such
random matrix quadratic forms $X_1 B_n X_1^T$
-- they are
studied in Section \ref{sec:preliminary-results} and Appendix
\ref{sec:appendix-b:-bai}. The latter presents the Bai-Yao CLT for
such quadratic forms, along with the martingale proof of
Baik-Silverstein.
Section \ref{sec:preliminary-results} adapts these results to the
spiked model setting.

Section \ref{sec:first-order-behavior} shows, by direct arguments,
that $\lhat_\nu \toas \rho_\nu$.
An eigenvalue perturbation lemma, Appendix
\ref{sec:appendix-c:-proof},  yields
$a_\nu \toas p_\nu$.
Convergence of $K(\lhat_\nu)$ to $(\rho_\nu/\ell_\nu) \Sigma$ and the
remaining proof of Theorem \ref{th:cosines} uses first order
properties of the quadratic forms $X_1 B_n X_1^T$ from Section
\ref{sec:preliminary-results}.

Section \ref{sec:eigenv-fluct} establishes the second order results.
After preliminary work on tightness at the $n^{-1/2}$ scale, the
eigenvalue CLT is derived from the quadratic form CLT.
The eigenvector CLT follows similarly, now using the quadratic error
bound provided by the eigenvector perturbation lemma of Appendix
\ref{sec:appendix-c:-proof}. 
Throughout, in our setting of simple spike eigenvalues,
an effort is made to present
the asymptotic variance formulas explicitly, showing Gaussian and
non-Gaussian contributions.
At a few points, for example the first part of Section
\ref{sec:eigenv-clt}, our 
exposition benefits from that in \cite{fjs18}, which addresses a more
general variance component setting with multiple levels of
variation.

\medskip
\textit{Notation.} \ 
$\| A \|$ denotes operator norm of $A$. 
For scalar $\ell$ we often write $A - \ell$ in place of $A - \ell I$.
If $X_n$ and $Y_n>0$ are sequences of random variables, we write
$X_n = O_{\rm a.s.}(Y_n)$ if there exists a variable $M < \infty$
a.s. such that $\limsup Y_n^{-1}|X_n| \leq M$ a.s. If $X_n$ is a
vector or matrix, then $|X_n|$ is replaced by $\|X_n\|$. 

\section{Outline of arguments}
\label{sec:outline-arguments}

Except where otherwise mentioned, we consider indices $\nu \leq m_0$,
so that $\ell_\nu > 1 + \sqrt \gamma$. 
We use the Schur complement decomposition
\begin{equation}  \label{eq:S-decomp}
  |S - \ell I_{m+p}| = |S_{22} - \ell I_{p}| \ |K(\ell) - \ell I_{m}|,
\end{equation}
where the $m \times m$ matrix
\begin{equation}
  \label{eq:Kell}
  K(\ell) = S_{11} + S_{12}(\ell I_{p} - S_{22})^{-1}S_{21}.
\end{equation}
Choose $\delta>0$ such that $\rho_\nu - b_\gamma > 3
\delta$. In view of
\eqref{eq:bai-yin},
we may when necessary confine attention to the event
$  E_{n\delta} = \{
                   \mu_1 \leq b_\gamma + \delta \},$
since $\mP(E_{n\delta}) \to 1$. 
When event $E_{n\delta}$ occurs, $\lhat_\nu$ is not an eigenvalue of
$S_{22}$ and so $|K(\lhat_\nu) - \lhat_\nu I_{m}| = 0$.

We begin with first-order behavior and the convergence of $\lhat_\nu$
to $\rho_\nu$. 
First express $K(t)$ as a random quadratic form by 
rewriting \eqref{eq:Kell} as
\begin{equation}
  \label{eq:Krho}
  K(t) = n^{-1} X_1 B_n(t) X_1^T,
\end{equation}
where the Woodbury identity is used in
\begin{equation} \label{eq:Bn-def}
  \begin{split}
  B_n  = B_n(t) 
       & = I_n + n^{-1}X_2^T(tI_p-n^{-1}X_2X_2^T)^{-1}X_2  \\
       & = t (t I_n-n^{-1}X_2^TX_2)^{-1}.  
  \end{split}
\end{equation}
Since $X_1$ is independent of $B_n(t)$, we will see from \eqref{eq:Ktn} that $K(t)$ has a
limit expressed in terms of the limiting companion distribution of the
eigenvalues of $n^{-1}X_2^TX_2$:
\begin{equation}
  \label{eq:Kt-conv}
  \begin{split}
  K(t) & \to
     (\aslim n^{-1} \tr B_n) \Sigma 
     = \int t(t-x)^{-1} \ssf_\gamma(dx) \cdot \Sigma 
     = -t \ssm(t;\gamma) \Sigma =: K_0(t;\gamma),
  \end{split}
\end{equation}
where ${\sf{m}}(t;\gamma)$ is the Stieltjes transform of the companion
distribution $\sf{F}_\gamma$. We then have,
\begin{equation*}
  K(t) - tI_m \tooas - t[\ssm(t ; \gamma) \Sigma + I_m].
\end{equation*}
In Appendix A we show that the equation $\ssm(t ; \gamma) = - 1/\ell_\nu$ has a
single root at $t = \rho_\nu$ in $(b_\gamma,\infty)$, and this leads to
the convergence  $\lhat_\nu \to \rho_\nu$. 
Thus $-\ssm(\rho_\nu ; \gamma) = 1/\ell_\nu$ and returning to
\eqref{eq:Kt-conv},
it can be shown (Section \ref{sec:first-order-behavior}) that
$K(\hat{\ell}_\nu) \to K_0(\rho_\nu ; \gamma) = 
(\rho_\nu/\ell_\nu) \Sigma$.

\bigskip
\textit{Eigenvector inconsistency (Section
  \ref{sec:eigenv-incons} for details).}
The cosine convergence of Theorem \ref{th:cosines} is a first-order
result, so it is natural to start there. 
We begin by reducing the eigenvector equation $S \mathfrak{u}_\nu =
\lhat_\nu \mathfrak{u}_\nu$ to the signal subspace. 
The partitioning (\ref{eq:part}) yields two equations
\begin{align*}
  S_{11} u_\nu + S_{12} v_\nu & = \lhat_\nu u_\nu \\
  S_{21} u_\nu + S_{22} v_\nu & = \lhat_\nu v_\nu,
\end{align*}
along with the normalization condition
\begin{equation*}
  u_\nu^T u_\nu + v_\nu^T v_\nu = 1.
\end{equation*}
From the second equation, $v_\nu = (\lhat_\nu I_p -
S_{22})^{-1}S_{21}u_\nu$ with the inverse defined
a.s. for large enough $n$.
Inserting this in the first equation yields 
$K(\lhat_\nu) u_\nu = \lhat_\nu u_\nu$, while putting
it into the third equation yields
\begin{equation*}
  u_\nu^T(I_m +Q_\nu)u_\nu = 1, \qquad \quad
  Q_\nu = Q(\lhat_\nu)
        = S_{12}(\lhat_\nu I_p - S_{22})^{-2} S_{21}.
\end{equation*}
Writing these two equations in terms of the signal-space normalized
sample eigenvectors $a_\nu = u_\nu/\|u_\nu\|$, we obtain
\begin{equation}
   \label{eq:eqn-pair}
 \begin{split}
  K(\lhat_\nu) a_\nu & = \lhat_\nu a_\nu  \\
  a_\nu^T(I_m +Q_\nu)a_\nu & = \|u_\nu\|^{-2}.     
 \end{split}
\end{equation}

The sample-to-population inner product can be rewritten
\begin{equation}
  \label{eq:cosin}
  \langle \mathfrak{u}_\nu, \mathfrak{p}_\nu \rangle
    =   \langle u_\nu, p_{\nu} \rangle
    = \| u_\nu \| \langle a_\nu, p_{\nu} \rangle.
  \end{equation}
The supercritical case of convergence result (\ref{eq:cos2}) will follow
from two facts
\begin{equation}
    \label{eq:pair-cge}
  a_\nu \tooas p_{\nu}, \qquad \quad 
  Q_\nu \tooas c(\rho_\nu) \Sigma,
\end{equation}
where $c(\rho_\nu)$ is defined at \eqref{eq:crho} and evaluated later at \eqref{eq:crhonu}. 
Indeed then, from \eqref{eq:eqn-pair}
\begin{equation*}
  \| u_\nu \|^{-2}
    \tooas p_\nu^T(I_m+c(\rho_\nu)\Sigma)p_{\nu}
    = 1 + c(\rho_\nu)\ell_\nu.
\end{equation*}
At (\ref{eq:slutsky}), it is computed that
$1 + c(\rho_\nu)\ell_\nu = \rho_\nu/(\ell_\nu \dot \rho_\nu)$.
Consequently, the last three displays yield
the supercritical part of Theorem \ref{th:cosines}:
\begin{equation*}
  \aslim \langle \mathfrak{u}_\nu, \mathfrak{p}_\nu \rangle^2
  = \aslim \|u_\nu \|^2
  = \ell_\nu \dot \rho_\nu/\rho_\nu.
\end{equation*}

For the subcritical case of \eqref{eq:cos2}, we show that $\| u_\nu \|
\to 0$ by showing that $\lmin (Q_\nu) \to \infty$.
This essentially follows from the divergence of the integral
\eqref{eq:crhonu} defining $c(\rho_\nu)$ when $\rho_\nu = b_\gamma$.

\bigskip
\textit{Eigenvalue fluctuations (Section \ref{sec:eigenv-fluct}).} \
Now consider second-order behavior.
The fluctuations of the quadratic form $X_1 B_n(t) X_1^T$ are
fundamental, so introduce
\begin{equation*}
R_n(\rho)
:= n^{-1/2} [X_1 B_n(\rho) X_1^T - \tr B_n(\rho) \Sigma], \qquad
R_{n\nu} := R_n(\rhonn).
\end{equation*}

In Section \ref{sec:tightness-properties}, we will establish the expansion
\begin{equation}
  \label{eq:main2}
  K(\rhonn)
     = (\rhonn/\ell_\nu)\Sigma
          + n^{-1/2} R_{n \nu} + O_p(n^{-1}),
\end{equation}
from \eqref{eq:gen(24)}.
The quantity of primary interest is $K(\lhat_\nu)$, and hence in Section \ref{sec:eigenv-clt}, it
will be shown to have the expansion
\begin{equation}
  \label{eq:delK}
  K(\lhat_\nu) - K(\rhonn) = - (\lhat_\nu - \rhonn)[c(\rho_\nu) \Sigma +
  o_p(1)].
\end{equation}
Combining these two displays and expanding $ p_\nu^T [K(\hat{\ell}_{\nu}) - \hat{\ell}_{\nu} I_m] p_{\nu}$
leads, Section \ref{sec:eigenv-clt}, 
 to a key representation for $\lhat_\nu - \rhonn$
\begin{equation}
  \label{eq:cltready}
  \sqrt n (\lhat_\nu-\rhonn)[1+c(\rho_\nu)\ell_\nu + o_p(1)]
      = p_\nu^T R_{n \nu} p_{\nu} + o_p(1).
\end{equation}
Asymptotic normality of $\lhat_\nu$ and Theorem \ref{th:eval-normality}
then follows from a central limit
theorem for $R_{n\nu} \toodist R^\nu$, Proposition \ref{prop:clt-r}
below. Especially,
the asymptotic variance of the right side is
\begin{equation}
  \label{eq:var-Rnunu}
  \text{\rm Var}(p_\nu^T R^\nu p_{\nu}) = 2 \theta_\nu \ell_\nu^2 +
  \omega_\nu \ell_\nu,
\end{equation}
where $\omega_\nu$ is the fourth cumulant of $p_\nu^T \xi$. 
The constants
$\theta_\nu, \omega_\nu$ and $1+c(\rho_\nu)\ell_\nu$ are all defined below; 
in the fixed $p$ case, they all reduce to $1$, and one gets the classical
CLT for separated sample covariance eigenvalues.

\bigskip
\textit{Eigenvector fluctuations (Section \ref{sec:eigenv-fluct-2}).} \ 
We regard the eigenvector equation $K(\lhat_\nu) a_\nu = \lhat_\nu a_\nu$
as a perturbation of the population equation
$(\rho_\nu/\ell_\nu)\Sigma p_{\nu} = \rho_\nu p_{\nu}$. 
A second order perturbation result for eigenvectors, Lemma
\ref{lemma:evec_perturb_bound}, yields a key representation 
\begin{equation}
     \label{eq:anu-decomp}
  a_\nu - p_{\nu} = - \mathcal{R}_{\nu n} D_{\nu} p_{\nu} + r_\nu,
\end{equation}
where 
$$ D_{\nu} := K(\lhat_\nu) - (\rhonn/\ell_\nu)\Sigma, $$ and 
the `reduced' resolvent is defined by 
\begin{equation}   \label{eq:rresolv}
  \cR_{\nu n} := (\ell_\nu/\rhonn) \sum_{\mu \neq \nu}^m (\ell_{\mu} -
  \ell_\nu)^{-1} p_{\mu} p_\mu^T.
\end{equation}
Lemma \ref{lemma:evec_perturb_bound} and proof of Lemma \ref{lem:tightness} say that $r_\nu
= O(\|D_\nu\|^2) = O_p(n^{-1})$. 
Further analysis allows us to write
\begin{equation}
  \label{eq:Dnud0}
  D_\nu 
   = n^{-1/2} R_{n\nu} + \hat{\delta}_{\nu n} \Sigma +
          o_p(n^{-1/2}), 
\end{equation}
where $ \hat{\delta}_{\nu n} := -[\hat{\ell}_\nu-\rhonn] c(\rho_\nu) 
= O_p(n^{-1/2})$.
A key remark is that $\cR_{\nu n} \Sigma p_{\nu} = 0$ and so 
we arrive at the representation 
\begin{equation}
   \label{eq:anu-fluct}
  \sqrt{n}(a_\nu - p_{\nu})
   = - \cR_{\nu n} R_{n\nu} p_{\nu} + o_p(1).
\end{equation}
Now we again apply the central limit theorem of Proposition
\ref{prop:clt-r} to $R_{n\nu}$. 
After some calculation, set out in Section \ref{sec:eigenv-fluct-2},
we obtain Theorem \ref{th:evec-normality}.



\section{Preliminary results}
\label{sec:preliminary-results}

\textbf{Random quadratic forms.} \ 
First we describe the first and second order behavior of the matrix of
quadratic forms $X_1 B_n X_1^T$, referring to results given in
detail in Appendix B.  

\begin{lemma}
   \label{lem:basi-b26}
For $i = 1, \ldots, n$, consider independent zero mean vectors
\begin{equation*}
  \begin{pmatrix}
    x_i \\ y_i
  \end{pmatrix} 
  \qquad \text{with}  \qquad
  \text{Cov}\begin{pmatrix}
    x_i \\ y_i
  \end{pmatrix} =
  \begin{pmatrix}
    1 & \rho \\
    \rho & 1
  \end{pmatrix}.
\end{equation*}
Suppose that $p \geq 1$ and that for $\ell = 4$ and $2p$ and all $i$
that $E |x_i|^\ell, \E |y_i|^\ell \leq \nu_\ell$.
Let $B$ be an $n \times n$ symmetric matrix,
$\bm{x} = (x_1, \ldots, x_n)^T$, and
$\bm{y} = (y_1, \ldots, y_n)^T$. Then
\begin{equation*}
  \E |\bm{x}^TB \bm{y}- \rho \tr B|^p
     \leq C_p[(\nu_4\tr B^2)^{p/2}+\nu_{2p}\tr B^p],
   \end{equation*}
   where the constant $C_p$ depends on $p$ only.
 \end{lemma}
\begin{proof}
The result is a direct consequence of \cite[Lemma B.26]{basi09}, which
applies to quadratic forms $\bm{x}^T B \bm{x}$ with $\bm{x}$ as in our
hypotheses. Indeed, define
$\bm{z} = (\bm{x}+\bm{y})/\sqrt{2+2\rho}$ and write
\begin{equation*}
  \bm{x}^TB \bm{y}
     =(1+\rho) \bm{z}^T B \bm{z} - \hf \bm{x}^T B \bm{x} - \hf
     \bm{y}^T B \bm{y},
\end{equation*}
and then use $(a+b+c)^p \leq 3^{p-1}(a^p+b^p+c^p).$ 
\end{proof}


Let $B_n$ be a sequence of $n \times n$ symmetric matrices.
Apply Lemma \ref{lem:basi-b26} to $B = n^{-1} B_n$ and then use
$\tr B_n^q \leq n \| B_n\|^q$ to conclude
\begin{equation*}
  \E | n^{-1}\bm{x}^TB_n \bm{y}- n^{-1}\rho \tr B_n|^p
    \leq C'_p \|B_n\|^p n^{-p/2}.
\end{equation*}
Hence if the moment conditions on $x_i$ and $y_i$ in Lemma
\ref{lem:basi-b26} hold for $\ell = 2p = 4+\delta$ and $\|B_n\|$ are
bounded, then it follows from Markov's inequality and the
Borel-Cantelli lemma that 
\begin{equation}
  \label{eq:asforconstB}
  n^{-1} \bm{x}^TB_n\bm{y} - n^{-1} \rho \tr B_n \tooas 0.
\end{equation}

We apply this in our Model M with \textit{random} matrices $B_n$,
independent of $X_1$: 
\begin{lemma}
  \label{lem:quad-form-as}
Assume Model M and suppose that $B_n = B_n(X_2)$ is a sequence of
random symmetric matrices for which
$\| B_n\|$ is $O_{\rm a.s.}(1)$.
Then
\begin{equation}
  \label{eq:quad-form-as}
  n^{-1}X_1B_n(X_2)X_1^T - n^{-1}\tr B_n(X_2)\Sigma \tooas 0.  
\end{equation}
\end{lemma}
\begin{proof}
At risk of being pedantic, we set out more precisely the assumption
on the random matrix sequence $B_n$. 
Let $(\Omega_i,\mathcal{F}_i,\mu_i)$ be probability spaces
supporting the sequences $\xi_1, \xi_2, \ldots$ for $i=1$ and
$\eta_1, \eta_2, \ldots$ for $i=2$. 
Recalling that $(\xi_i)$ and $(\eta_i)$ are assumed independent, let
$(\Omega,\mathcal{F},\mu) = (\Omega_1 \times \Omega_2, \mathcal{F}_1
\times \mathcal{F}_2, \mu_1 \times \mu_2)$ be the product space.

We assume that for $\mu_2$-almost all $\omega_2 \in \Omega_2$, there exists
$n(\omega_2)$ such that for $n > n(\omega_2)$ the sequence
$B_n(\omega_2)$ is defined and $\|B_n(\omega_2)\|$ is bounded.
[In view of the Bai-Yin convergence (\ref{eq:bai-yin}) and the remark
following (\ref{eq:as-fnal}),
this condition will typically hold in our examples below 
in which $B_n$ is a function of $X_2$
and resolvent matrices $(tI_p - n^{-1}X_2X_2^T)^{-1}$
or $(tI_n - n^{-1}X_2^TX_2)^{-1}$ for $t > b_\gamma$.]

Making the dependence on sample point $\omega = (\omega_1,\omega_2)$
explicit, let 
$$Y_n(\omega) = n^{-1}[X_1(\omega_1)^TB_n(\omega_2)X_1(\omega_1) -
\tr B_n(\omega_2) \Sigma].$$
Let $E = \{ \omega : Y_n(\omega) \to 0\}$ and for each $\omega_2$
define the section $E_{\omega_2} = \{\omega_1: Y_n(\omega_1,\omega_2)
\to 0\}$.
Now apply Lemma \ref{lem:basi-b26} and (\ref{eq:asforconstB}) 
with $B_n = B_n(\omega_2)$ to show
that $\mu_1(E_{\omega_2}) = 1$ for $\mu_2$-almost all $\omega_2$. 
Fubini's theorem now shows that $\mu(E) = 1$, as required. 
\end{proof}

Before illustrating this result, we record a 
a useful consequence of weak convergence
of $\ssf_n$ to $\ssf_\gamma$ combined with
the almost sure convergence of
extreme eigenvalues \eqref{eq:bai-yin}.
If $f_n \to f$ uniformly as continuous functions on the closure $I$ of a
bounded neighborhood of the support of $\ssf_\gamma$,  then
\begin{equation} \label{eq:as-fnal}
  \int f_n(x) \ssf_n(dx) \stackrel{\rm a.s.}{\longrightarrow} 
  \int f(x) \ssf_\gamma(dx).
\end{equation}
[If $\text{supp}({\ssf}_n)$ is not contained in $I$, then the left side
integral may not be defined. However, such an event occurs  for
at most finitely many $n$ with probability one].

Let us illustrate how we apply this result.
Suppose $t_n \to t > b_\gamma + 2 \delta$, and let $r$ be a positive
integer. 
The functions $f_n(x) = t_n^r(t_n-x)^{-r}$ converge uniformly to 
$f(x) = t^r(t-x)^{-r}$ for $x \leq b_\gamma + \delta$.
Let $B_n(t)$ be given  as in \eqref{eq:Bn-def}, and let
$\tilde E_{n\delta} = \{ \mu_1 < b_\gamma + \delta, t_n \geq b_\gamma + 2 \delta\}$. 
Now set $B_n = \mathbbm{1}( \tilde E_{n\delta} ) B_n(t)$ and 
note that $\mathbbm{1}( \tilde E_{n\delta} ) \toas 1$. 
From  (\ref{eq:as-fnal}), the normalized traces
\begin{equation}  
     \label{eq:lln}
  n^{-1} \tr B_n^r(t_n) 
    = I_n \cdot n^{-1} \sum_{i=1}^n t_n^r (t_n-\mu_i)^{-r}
    \tooas \int t^r (t-x)^{-r} {\sf{F}}_\gamma(dx)
\end{equation}
and are $O_{\rm a.s.}(1)$. 
Similarly $\| B_n \| = t_n(t_n-\mu_1)^{-1}$
is $O_{\rm a.s.}(1)$.
From Lemma \ref{lem:quad-form-as}
we may then conclude that
\begin{equation}
  \label{eq:Ktn}
K(t_n) \tooas K_0(t;\gamma). 
\end{equation}


The fluctuations of the matrix random quadratic forms $X_1 B_n X_1^T$
are asymptotically 
Gaussian. This is formalized in the next result, itself a 
direct consequence of a foundational CLT for vectors of
scalar bilinear forms recalled as Theorem \ref{th:bybs} in Appendix
B. 

\begin{proposition}
  \label{prop:matCLT}
Suppose that the columns of an $m \times n$ matrix $X_1$ are
distributed i.i.d. as $\xi$ with mean $0$, covariance matrix
$\Sigma = (\sigma_{jk})$ and fourth order cumulant tensor
$\kappa_{jklm} = \E[\xi_j \xi_k \xi_l \xi_m]
-\sigma_{jk}\sigma_{lm}-\sigma_{jl}\sigma_{km}-\sigma_{jm}\sigma_{kl}$.
Let $B_n$ be a sequence of $n \times n$ real symmetric random
matrices, independent of $X_1$.
Assume that $\|B_n\| \leq a$ for all $n$,
and that
\begin{equation*}
  \theta = \plim n^{-1} \tr B_n^2,\qquad \quad
  \omega = \plim n^{-1} \sum_{i=1}^n b_{n,ii}^2
\end{equation*}
are both finite.
Then the centered and scaled matrix
\begin{equation*}
  R_n = n^{-1/2}[X_1 B_n X_1^T - (\tr B_n) \Sigma] 
      \ \stackrel{\mathcal{D}}{\to} \ R,
\end{equation*}
a symmetric $m \times m$ Gaussian matrix having entries $R_{ij}$ with
mean zero and 
covariances given by
\begin{equation}
  \label{eq:covRijipjp}
  \text{cov}(R_{jk},R_{j'k'}) 
    = \theta [\sigma_{jk'} \sigma_{j'k} + \sigma_{jj'}
     \sigma_{kk'}] + 
        \omega \, \kappa_{jkj'k'}.
\end{equation}
\end{proposition}
If $B_n = I_n$, then $\omega = \theta = 1$ and we recover the classical CLT for
sample covariance matrices cited, for example, in \citet[Theorem
1.2.17 and p. 42]{muir82}.
The result here is a version of \cite[Prop 3.1]{bayo08} in which
$B_n$ is not yet 
specialized, and the spectral norm condition is added. Otherwise the
proof is as in [BY], but is included for completeness. 

\begin{proof}
  We turn the matrix quadratic form $X_1 B_{n} X_1^T$ into a vector of
  bilinear forms (for use in Theorem \ref{th:bybs}) by a common device also
  used in [BY].
Use an index $\ell$ for the $L = m(m+1)/2$ pairs $(j,k)$ with $1 \leq j \leq
k \leq m$. \footnote{The index $\ell$ is not to be confused with eigenvalues
$\ell_\nu!$}
Build the random vectors $(x,y)$ for Theorem \ref{th:bybs} from $\xi$ as
follows:  if $\ell = (j,k)$ then set
$x_\ell = \xi_j$ and $y_\ell = \xi_k$.
In the resulting covariance matrix $\Gamma$ for $(x,y)$, if also
$\ell'=(j',k')$, we have
\begin{equation}
  \label{eq:Gam1}
  \Gamma^{xy}_{\ell \ell'} 
   = \E x_\ell y_{\ell'} 
   = \E \xi_j \xi_{k'} 
   = \sigma_{jk'},
\end{equation}
so that, in particular, $\rho_\ell = \Gamma^{xy}_{\ell \ell} 
= \sigma_{jj}$. Similarly,
\begin{equation}
  \label{eq:Gam2}
  \begin{split}
  \Gamma^{xx}_{\ell \ell'} = \sigma_{jj'}, \qquad
  \Gamma^{yx}_{\ell \ell'} & = \sigma_{kj'}, \qquad
  \Gamma^{yy}_{\ell \ell'} = \sigma_{kk'}, \quad \textrm{and} \\
  K_{\ell \ell'} & = \kappa_{jkj' k'}.    
  \end{split}
\end{equation}
Component $R_{\nu n,jk}$ corresponds to component $Z_\ell$ in Theorem
\ref{th:bybs}, and we conclude that $R_{\nu n} \stackrel{\mathcal{D}}{\to} R$.
The latter is Gaussian mean zero with covariance matrix $D$ given by
\eqref{eq:eqD}, so that 
$\text{cov}(R_{jk},R_{j'k'}) = \theta J_{\ell \ell'} + \omega K_{\ell
  \ell'}$,
where $J_{\ell \ell'}$ is defined at \eqref{eq:Jdef}.
Substituting \eqref{eq:Gam1} and \eqref{eq:Gam2}
yields \eqref{eq:covRijipjp}. 
\end{proof}

\bigskip
We will apply Proposition \ref{prop:matCLT} in the setting of Model M
to obtain a central limit theorem for $K(\rhonn)$ as defined at
(\ref{eq:Krho}). 

\begin{proposition}
  \label{prop:clt-r}
Assume Model M and define $\rhonn$ by (\ref{eq:rhodef}) and 
$B_n(\rhonn)$ by (\ref{eq:Bn-def}). Then
\begin{equation} \label{eq:Rnnu-cge}
  R_{n\nu} = n^{-1/2}[X_1B_n(\rhonn)X_1^T - \tr B_n(\rhonn) \Sigma]
           \stackrel{\mathcal{D}}{\longrightarrow} R^\nu,
\end{equation}
a symmetric Gaussian random matrix having entries $R_{jk}^\nu$ with
mean zero and covariance given by (\ref{eq:covRijipjp}), along with
$\theta$ and $\omega$ given by
\begin{equation}
  \label{eq:th-om-nu}
  \theta_\nu = \frac{1}{\dot \rho_\nu}
  \left(\frac{\rho_\nu}{\ell_\nu}\right)^2 
  = \frac{(\ell_\nu-1+\gamma)^2}{(\ell_\nu-1)^2-\gamma}, 
    \qquad \quad 
    \omega_\nu = \left(\frac{\rho_\nu}{\ell_\nu}\right)^2 
    = \frac{(\ell_\nu-1+\gamma)^2}{(\ell_\nu-1)^2}.
\end{equation}
In addition,
\begin{align}
  \label{eq:Rvar}
  \text{\rm Var}(p_\nu^T R^\nu p_{\nu}) & = 2 \theta_\nu \ell_\nu^2 +
  \omega_\nu [\mathcal{P}^\nu, \kappa], \\
  \label{eq:Rcolcov}
  \text{\rm Cov}(R^\nu p_{\nu}) & = \theta_\nu \ell_\nu \Sigma
      + \theta_\nu \ell_\nu^2  p_{\nu} p_\nu^T +   \omega_\nu K^\nu,
\end{align}
where $[\mathcal{P}^\nu, \kappa]$ and  $K^\nu$  are defined at \eqref{eq:kapmumu}--\eqref{eq:innerproduct}.
\end{proposition}

\textit{Special cases.} \ 
If the components of $\xi$ are independent, then
$\Sigma  = \textrm{diag}(\ell_\nu)$ with $p_{\nu}$ the coordinate basis
vectors, 
$K^\nu = [\mathcal{P}^\nu, \kappa] p_{\nu} p_\nu^T$ and 
$[\mathcal{P}^\nu, \kappa] = \E \xi_\nu^4 - 3 \ell_\nu^2$ is the fourth cumulant of
$\xi_\nu$.
If $\xi$ is Gaussian, then each of $\kappa_{jklm}, K^\nu$ and
$[\mathcal{P}^\nu, \kappa]$ vanish.

\begin{proof}
  The matrix $B_n(\rhonn) = \rhonn(\rhonn I_n - n^{-1} X_2^T
  X_2)^{-1}$ is independent of $X_1$.
  Define $\check E_{n\delta} = \{ \mu_1(n^{-1}X_2^T X_2) \leq b_\gamma + \delta,
  \rhonn > \rho_\nu-\delta \}$ and apply Proposition \ref{prop:matCLT}
  to $B_n = B_n(\rhonn) \mathbbm{1}( \check E_{n\delta} )$, which is also independent of $X_1$.
  We have
\begin{equation*}
  \| B_n \| = 
  \| B_n(\rhonn) \mathbbm{1}( \check E_{n\delta} ) \|
  \leq \rhonn/(\rhonn-\mu_1) \leq \delta^{-1}(\rho_\nu+\delta)
  = a, 
\end{equation*}
say. The CLT of Proposition \ref{prop:matCLT} applies to
$R_n =  \mathbbm{1}( \check E_{n\delta} ) R_{n \nu}$. Since $ \mathbbm{1}( \check E_{n\delta} ) \toas 1$, we obtain \eqref{eq:Rnnu-cge}.
It remains to evaluate
$\theta_\nu$ and $\omega_\nu$: these are derived in Appendix A at
(\ref{eq:thetanu}) and Lemma \ref{lem:omegaval}; the presence of $\mathbbm{1}_{\check E_{n\delta}} \toas 1$ is immaterial.

To verify \eqref{eq:Rcolcov}, we calculate using \eqref{eq:covRijipjp}
\begin{align*}
  \textrm{cov}[(R^\nu p_{\nu})_{j},(R^\nu p_{\nu})_{j'}]
    & = \sum_{k,k'} p_{\nu,k} \, \textrm{cov}(R^\nu_{jk},R^\nu_{j'k'})
      p_{\nu,k'} \\
    & = \sum_{k,k'} \theta_\nu p_{\nu,k} [\sigma_{jk'} \sigma_{j'k} + \sigma_{jj'}
     \sigma_{kk'}] p_{\nu,k'} + \omega_\nu p_{\nu,k} \kappa_{jj'kk'}
      p_{\nu,k'} \\
    & = \theta_\nu [(\Sigma p_{\nu})_j (\Sigma p_{\nu})_{j'} +
      \sigma_{jj'} p_{\nu}^T \Sigma p_{\nu}] + \omega_\nu K^\nu_{jj'},
\end{align*}
which reduces to \eqref{eq:Rcolcov}.
Since $\textrm{Var}(p_{\nu}^T R^\nu p_{\nu}) = p_{\nu}^T \textrm{Cov}(R^\nu
p_{\nu}) p_{\nu}$, variance formula \eqref{eq:Rvar} now follows from
\eqref{eq:Rcolcov}. 
\end{proof}

\section{First Order Behavior}
\label{sec:first-order-behavior}


\subsection{Eigenvalues}
\label{sec:eigenvalues}

\textit{Proof that $\lhat_\nu \toas \rho_\nu$
for $\nu \leq m_0$.} \ \ 
Although established in \citet{basi06},  in our
setting of simple supercritical $\ell_\nu$ there is a more direct proof
that also introduces tools needed below.

For $\rho > b_{\gamma}$, $\lambda_{0 \nu}(\rho)= - \rho \ssm(\rho;\gamma)\ell_\nu - \rho$ is the $\nu^{\text{th}}$ eigenvalue of $K_0(\rho;\gamma) -\rho I_m$. 
Observe that
\begin{equation*}
  \partial_\rho \lambda_{0 \nu}(\rho) = - 1 - \ell_{\nu} \int x(\rho-x)^{-2}
  {\sf{F}}_{\gamma}(dx) < -1.
\end{equation*}
Since $\lambda_{0 \nu}(\rho_\nu) = 0$, 
if we define $\rho_{\nu \pm} = \rho_\nu \pm \epsilon$, this entails that
\begin{equation*}
  \lambda_{0 \nu}(\rho_{\nu -}) \geq \epsilon, \qquad
  \lambda_{0 \nu}(\rho_{\nu +}) \leq -\epsilon.
\end{equation*}
Now let $\lambda_\nu(\rho)$ be the $\nu^{\text{th}}$ eigenvalue of $K(\rho) - \rho I_m$. 
A standard eigenvalue perturbation bound yields
\begin{equation}
  \label{eq:evalpbd}
  |\lambda_\nu(\rho) - \lambda_{0 \nu}(\rho)|
  \leq \| K(\rho)-K_0(\rho;\gamma) \| \tooas 0.
\end{equation}
So for $n \geq n_0(\omega,\epsilon)$, where $\omega$ is a sample from $\Omega$ defined in Lemma \ref{lem:quad-form-as}, and $\epsilon > 0$, we must have, for each $\nu \leq
m_0$, 
\begin{equation*}
  \lambda_{\nu}(\rho_{\nu -}) \geq \hf \epsilon, \qquad
  \lambda_{\nu}(\rho_{\nu +}) \leq - \hf \epsilon.
\end{equation*}
Since $\lambda_\nu(\rho)$ is continuous in $\rho$, it must be that
$\lambda_\nu(\rho_{\nu *}) = 0$ for some $\rho_{\nu *} \in (\rho_{\nu -},\rho_{\nu
  +})$.
From \eqref{eq:S-decomp}, $\rho_{\nu *}$ is an eigenvalue of $S$, and
since $\lhat_{m_0+1} \to b_\gamma$ and all supercritical eigenvalues are simple,
it must be that $\rho_{\nu *} =\lhat_\nu$, by taking $\epsilon < \min \{ \rho_i - \rho_{i+1} \}_{i = 1, \cdots, m_0}$.

Taking account also of \eqref{eq:bai-yin}, 
we have  shown that with probability one, there
exists $n_0(\omega)$ such that for all $n > n_0(\omega)$, 
\begin{equation} \label{eq:intervals}
  \mu_1(\omega) < b_\gamma + \epsilon_0, \qquad
  |\hat{\ell}_\nu(\omega) - \rho_\nu| < \epsilon_0, \quad
  \nu = 1, \ldots, m_0.
\end{equation}
In particular, $|K(\hat{\ell}_\nu) - \hat{\ell}_\nu I_m | = 0$ and
$\lhat_\nu \toas \rho_\nu$.



\medskip
\textit{Proof that $K(\lhat_\nu)  \ \tooas \ (\rho_\nu/\ell_\nu)
  \Sigma$.} \ \ 
Now we argue that 
\begin{equation}
   \label{eq:Kfirstorder}
  \Delta_{\nu n} = K(\lhat_\nu) - \hat{\ell}_\nu I_m 
     \ \tooas \ (\rho_\nu/\ell_\nu) \Sigma - \rho_\nu I_m.
\end{equation}
Begin with the decomposition
\begin{equation}
   \label{eq:Delnun}
  \Delta_{\nu n} 
    = K(\rhonn)-\lhat_\nu I +K(\lhat_\nu) -K(\rhonn).
\end{equation}
We have seen at (\ref{eq:Ktn}) 
that $K(\rhonn) \toas (\rho_\nu/\ell_\nu) \Sigma$, 
and then that 
$\lhat_\nu \toas \rho_\nu$.

Now we show that $K(\hat{\ell}_\nu)-K(\rhonn) \toas 0$.
Recall that $C_n = n^{-1} X_2^T X_2$ and 
introduce the resolvent\footnote{Most--but not all--authors use 
  $(C_n - \ell I)^{-1}$, 
but for us this form yields positive
  definite matrices.} notation
$R(\ell) = (\ell I_n - C_n)^{-1}$
so that from \eqref{eq:Krho} and \eqref{eq:Bn-def},
$K(\rho) = n^{-1} X_1 \rho R(\rho) X_1^T$ and 
$B_n(\rhonn) = \rhonn R(\rhonn)$.
From the resolvent identities
\begin{align}
  \label{eq:res-id}
  R(\ell) - R(\rho) & = -(\ell-\rho) R(\ell)R(\rho), \\
   \ell R(\ell) - \rho R(\rho)
    & = -(\ell - \rho) C_n R(\ell) R(\rho),
\end{align}
after noting that $\ell R(\ell) = C_n R(\ell) + I$, we obtain
\begin{align}
  \label{eq:Kdiff}
  K(\lhat_\nu) - K(\rhonn)
  & = -(\lhat_\nu-\rhonn) n^{-1} X_1 C_n R(\lhat)R(\rhonn) X_1^T \\
  & = -(\lhat_\nu-\rhonn) n^{-1} X_1 C_n R^2(\rhonn) X_1^T \notag \\
  & \quad 
  + (\lhat_\nu-\rhonn)^2 n^{-1} X_1 C_n R(\lhat)R^2(\rhonn)
    X_1^T.   \label{eq:Kdiff2}
\end{align}

For $n > n_0(\omega)$, we have $\lhat_\nu > \rho_\nu -\epsilon_0$ from
\eqref{eq:intervals}, and so 
$R(\lhat_\nu) \prec R(\rho_\nu-\epsilon_0)$ in the ordering of
nonnegative definite matrices and 
consequently, for $k \in \mathbb{N},$
\begin{equation} \label{eq:B_nk_quad_ineq}
  n^{-1} X_1 B_{nk}(\lhat_\nu,\rhonn) X_1^T
  \prec
  n^{-1} X_1 B_{nk}(\rho_\nu-\epsilon_0,\rhonn) X_1^T,
\end{equation}
where for later reference we define
$B_{nk}(\ell,\rho) = C_n R(\ell)R^k(\rho)$.
As
\begin{equation}  \label{eq:B_nk_spectral_norm}
\| B_{nk} (\rho_{\nu} - \epsilon_0, \rho_{\nu n})\| = \frac{\mu_1}{(\rho_{\nu} - \epsilon_0 - \mu_1) (\rho_{n\nu} - \mu_1 )^k} \quad \rm a.s.,
\end{equation}
for $n > n_0(\omega)$, 
we can apply Lemma \ref{lem:quad-form-as}
to $B_n = B_{nk}(\rho_\nu-\epsilon_0,\rho_{\nu n})$, which implies that 
$n^{-1} X_1 B_{nk}(\lhat_\nu,\rhonn) X_1^T = O_{\rm a.s.}(1)$ and 
hence that
\begin{equation}
  \label{eq:delK-firstorder}
     K(\hat{\ell}_{\nu}) - K(\rho_{\nu n}) = (\hat{\ell}_{\nu} - \rho_{\nu n }) \cdot n^{-1} X_1 B_{n1} (\hat{\ell}_{\nu}, \rho_{\nu n}) X_1^T
     = (\hat{\ell}_{\nu} - \rho_{\nu n }) \cdot O_{\rm a.s.}(1) \overset{\rm a.s.}{\rightarrow} 0.
   \end{equation}

\subsection{Eigenvectors}
\label{sec:eigenv-incons}

\textit{Supercritical case, $\ell_\nu > 1 + \sqrt \gamma.$}
We show \eqref{eq:pair-cge}, firstly $a_\nu \toas p_{\nu}$.
In view of the first order approximations \eqref{eq:Baik-S} and
\eqref{eq:Kfirstorder}, we regard the eigenvector equation 
$K(\hat{\ell}_\nu) a_\nu = \lhn a_\nu$ as a perturbation 
of $(\rhonn/\ell_\nu) \Sigma p_{\nu} = \rhonn p_{\nu}$. 
Thus, in order to use Lemma \ref{lemma:evec_perturb_bound}, we identify $r$
with $\nu$, $A$ with $(\rhonn/\ell_\nu) \Sigma$ and 
$B$ with $ D_\nu = K(\lhn) - (\rhonn/\ell_\nu) \Sigma.$
The eigenvalue separation is
$  \delta_\nu(A) 
    = (\rhonn/\ell_\nu) \min_{k\neq \nu} \{|\ell_\nu-\ell_k|\} \geq
    \delta_\nu > 0, $
say, and we identify the reduced resolvent $H_r(A)$ with
$\mathcal{R}_{\nu n}$ in (\ref{eq:rresolv}).
To bound $\| D_\nu \|$, use the decomposition
\begin{equation}
  \label{eq:Dnuto0}
  D_\nu  = K(\rhonn) - (\rhonn/\ell_\nu)\Sigma + K(\hat{\ell}_\nu) -
  K(\rhonn) \tooas 0,
\end{equation}
which follows from (\ref{eq:Ktn}) with $t_n = \rhonn, t = \rho_{\nu}$ and from
(\ref{eq:delK-firstorder}). 
Then identify $\mathbf{p}_r(A+B)$ and $\mathbf{p}_r(A)$ with $a_\nu$
and $p_{\nu}$  
respectively. Lemma \ref{lemma:evec_perturb_bound} yields the
important decomposition (\ref{eq:anu-decomp}), namely
$a_\nu - p_{\nu} = - \mathcal{R}_\nu D_\nu p_{\nu} + r_\nu$,
and the bound $\| r_\nu \| = O(\|D_\nu\|^2)$. 
Using (\ref{eq:Dnuto0}), we conclude that $a_\nu - p_{\nu} \toas 0$.

\medskip
Now consider $Q_\nu \to c(\rho_\nu) \Sigma$ a.s..
Write $R_p(t) = (tI_p-S_{22})^{-1}$. Then
\begin{align*}
  Q_\nu & = S_{12}R_p^2(\rho_\nu)S_{21} +
          S_{12}[R_p^2(\lhat_\nu)-R_p^2(\rho_\nu)]S_{21} \\
        & = n^{-1}X_1 \check{B}_{n1} X_1^T 
            -(\lhat_\nu-\rho_\nu) n^{-1}X_1 \check{B}_{n2}X_1^T
          = Q_{\nu 1} + Q_{\nu 2},
\end{align*}
say, where $\check{B}_{n1} = n^{-1}X_2^TR_p^2(\rho_\nu)X_2$,
and from a variant
of the resolvent identity, namely
$$ R^2_{p} (\ell) - R^2_{p} (\rho) 
= - (\ell - \rho)( ( \ell + \rho ) I_p - 2 S_{22}) R^2_{p} (\ell) R^2_{p} (\rho),$$
we obtain
\begin{equation*}
  \check{B}_{n2}
     = n^{-1}X_2^T[(\lhat_\nu+\rho_\nu) I_p -2S_{22}]
     R_p^2(\lhat_\nu)R_p^2(\rho_\nu) X_2 = O_{\rm a.s.}(1),
\end{equation*}
since 
$$\| \check{B}_{n2} \| = 
\frac{(\hat{\ell}_{\nu} + \rho_{\nu} - 2 \mu_1) \mu_1}{(\hat{\ell}_{\nu}-\mu_1)^2 (\rho_{\nu}-\mu_1)^2 } \quad \rm a.s.. $$
Since $\lhat_\nu \toas \rho_\nu$, we conclude that $Q_{\nu 2} \toas
0$. 
For $Q_{\nu 1}$ we use Lemma \ref{lem:quad-form-as}, and note that
\begin{equation*}
  n^{-1}\tr \check{B}_{n1}
    = n^{-1} \tr R^2(\rho_\nu) S_{22} = n^{-1} \tr B_{n1}(\rho_\nu,\rho_\nu),
  \end{equation*}
where the second equality uses the fact that
$R_p^2(\rho) S_{22}$ and $C_n R_n^2(\rho)$ have the same nonzero
eigenvalues $(\rho-\mu_i)^{-2} \mu_i$ for $i = 1, \ldots p \wedge n$.   
  Thus $Q_{\nu 1} \toas c(\rho_\nu)\Sigma$, exactly as at 
(\ref{eq:crho}).
This completes the proof of \eqref{eq:pair-cge} and hence the
supercritical part of  Theorem \ref{th:cosines}.

\medskip
\textit{Subcritical case, $\ell_\nu \leq 1 + \sqrt \gamma$}. \ 
To show that $\langle u_\nu, p_{\nu} \rangle = \| u_\nu \| \langle
a_\nu, p_{\nu} \rangle \toas 0$, it suffices, from \eqref{eq:eqn-pair},
to show that $a_\nu^T Q_\nu a_\nu \toas \infty$. 
We will establish this by showing that $\lambda_{\rm \min}(Q_\nu)
\toas \infty$. 
The approach uses a regularized version
\begin{equation*}
  Q_{\nu \epsilon}(t) = S_{12}[(t I_p - S_{22})^2 + \epsilon^2
  I_p]^{-1} S_{21},
\end{equation*}
for $\epsilon > 0$. Observe that $Q_\nu = Q_\nu(\lhat_\nu) \succ
Q_{\nu \epsilon} (\lhat_\nu)$, so that
\begin{equation*}
  \liminf \lmin(Q_\nu)
  \geq \liminf \lmin(Q_{\nu \epsilon}(\lhat_\nu))
  = \liminf \lmin(Q_{\nu \epsilon}(b_\gamma) + \Delta_{\nu \epsilon}),
\end{equation*}
where $\Delta_{\nu \epsilon} := Q_{\nu \epsilon}(\lhat_\nu)- Q_{\nu \epsilon}(b_\gamma) $ 
(Recall that $\hat{\ell}_\nu \toas b_\gamma$).
We will show that
$\Delta_{\nu \epsilon} \toas 0$,
and
\begin{equation}
  \label{eq:qnuep}
  Q_{\nu \epsilon}(b_\gamma) \toas
  \int x[(b_\gamma-x)^2 + \epsilon^2]^{-1} \ssf_\gamma(dx) \cdot \Sigma
  = c_\gamma(\epsilon) \Sigma,
\end{equation}
say. Since $\lmin(\cdot)$ is a continuous function on $m \times m$
matrices, we conclude that
\begin{equation}
  \label{eq:cgamma}
  \liminf \lmin(Q_\nu) \geq c_\gamma(\epsilon) \lmin(\Sigma),
\end{equation}
and since 
$c_\gamma(\epsilon) \geq  c(b_{\gamma} + \epsilon)$ and $c(b_{\gamma} + \epsilon) \nearrow \infty$
as $\epsilon \searrow
0$ by \eqref{eq:crhonu}, we obtain $\lambda_{\rm \min}(Q_\nu)
\toas \infty$. 

Let us write $Q_{\nu \epsilon}(t) = n^{-1} X_1 \check B_{n \epsilon}(t) X_1^T$, with
\begin{align*}
  \check B_{n \epsilon} (t) & =
           n^{-1} X_2^T[(tI_p-n^{-1}X_2X_2^T)^2+\epsilon^2 I_p]^{-1}X_2 \\
         & = H \diag \{ f_\epsilon(\mu_i,t) \} H^T,
\end{align*}
if we write the singular value decomposition of $n^{-1/2}X_2 = V
\mathcal{M}^{1/2} H^T$, with $\mathcal{M}=\diag (\mu_i)_{i=1}^p$, and
define $f_\epsilon(\mu,t) = \mu[(t-\mu)^2+\epsilon^2]^{-1}$.
Evidently $\| \check B_{n \epsilon}(t) \| \leq \epsilon^{-2}\mu_1$ is bounded a.s.
Thus Lemma \ref{lem:quad-form-as} may be applied to $Q_{\nu
  \epsilon}(b_\gamma)$, and since
\begin{equation*}
  n^{-1} \tr \check B_{n \epsilon}(b_\gamma) \toas \gamma \int x
  [(b_\gamma-x)^2+\epsilon^2]^{-1} F_\gamma(dx) = c_\gamma(\epsilon)
\end{equation*}
from \eqref{eq:companion}, our claim \eqref{eq:qnuep} follows.

Consider now $\Delta_{\nu \epsilon}$. 
Fix $a \in \mathbb{R}^m$ such that $\|a\|_2 = 1$, and set $b = n^{-1/2}H^T X_1^T a$.
We have
\begin{equation*}
a^T \Delta_{\nu \epsilon} a
= \sum_{i=1}^p b_i^2 [f_\epsilon(\mu_i,\lhat_\nu) -
f_\epsilon(\mu_i,b_\gamma)]. 
\end{equation*}
Since $|\partial f_{\epsilon}(\mu, t) / \partial t | = |2 \mu (t-\mu)| / [(t-\mu)^2 + \epsilon^2]^2 \leq \mu / \epsilon^3$ for $\mu, \epsilon > 0$ by AM-GM, we have
\begin{equation*}
| a^T \Delta_{\nu \epsilon} a | 
\leq \mu_1  \epsilon^{-3} |\hat{\ell}_{\nu} - b_{\gamma}| \cdot \|b\|_2^2
= \mu_1  \epsilon^{-3} |\hat{\ell}_{\nu} - b_{\gamma}| a^T S_{11} a 
\leq \mu_1  \epsilon^{-3} |\hat{\ell}_{\nu} - b_{\gamma}| \lhat_1
\overset{\rm a.s.}{\rightarrow}  0,
\end{equation*}
from Cauchy's interlacing inequality for eigenvalues of symmetric matrices and the first order behavior of $\lhat_\nu$ given in \eqref{eq:Baik-S}.
Hence $\Delta_{\nu \epsilon} \toas 0$ and the proof of
\eqref{eq:cgamma} and hence of the subcritical part of Theorem
\ref{th:cosines} is done.

\section{Second Order Results}
\label{sec:eigenv-fluct}

\subsection{Tightness properties}
\label{sec:tightness-properties}

We first establish a decomposition for $K(\rho) - K_0(\rho;\gamma_n)$.
Set $g_\rho(x) = \rho (\rho - x)^{-1}$ and write
\begin{equation*}
  \tr B_n(\rho) = \sum_{i=1}^n \rho (\rho-\mu_i)^{-1}
       = \sum_{i=1}^n g_\rho(\mu_i).
\end{equation*}
\citet{basi04} establish a central limit theorem for the unnormalized
sums
\begin{equation*}
  G_n(g) = \sum_{i=1}^n g(\mu_i) - n\int g(x) {\sf{F}}_{\gamma_n}(dx).
\end{equation*}
[Here it is important that the $n$-dependent deterministic
approximation ${\sf{F}}_{\gamma_n}$ is used rather than
${\sf{F}}_{\gamma}$.]
Recall from \eqref{eq:Kt-conv} that
$K_0(\rho ;\gamma_n) = - \rho \ssm(\rho;\gamma_n) \Sigma$
with  $-\ssm(\rho;\gamma_n) = \int (\rho-x)^{-1}
{\sf{F}}_{\gamma_n}(dx).$
Combining these remarks, we obtain
\begin{align}
  K(\rho)- K_0(\rho;\gamma_n)
   & = n^{-1}[X_1 B_n(\rho) X_1^T - \tr B_n(\rho) \Sigma] \notag \\
   & \qquad + \rho n^{-1} \bigg[\sum_{i=1}^n (\rho-\mu_i)^{-1}
                 -n \int (\rho-x)^{-1} \ssf_{\gamma_n}(dx) \bigg]
     \Sigma \notag \\ 
   & = n^{-1/2} R_n(\rho) +  n^{-1} G_n(g_\rho) \Sigma. \label{eq:Wdecomp}
\end{align}
Proposition \ref{prop:clt-r} tells us that
$R_n(\rho) = O_p(1)$ in the first term.
In the second term, the functions $g_\rho$ are analytic on a
neighborhood of $[0,b_\gamma]$ in $\mC$.
The central limit theorem of \cite{basi04} shows 
in particular that $G_n(g_\rho) = O_p(1)$.
In summary, we obtain for each
$\rho \geq b_\gamma + 3\delta$,
\begin{equation} \label{eq:gen(24)}
  K(\rho) - K_0(\rho;\gamma_n) = n^{-1/2} R_n(\rho)  +  O_p(n^{-1}) = O_p(n^{-1/2}).
\end{equation}

\begin{lemma}
  \label{lem:tightness}
  Assume that Model M holds and that $\ell_\nu > 1 + \sqrt
  \gamma$. For some $b > \rho_1$, let $I$ denote the interval
  $[b_\gamma+3 \delta, b]$. Then
  \begin{gather}
    \lhat_\nu-\rhonn = O_p(n^{-1/2}),      \label{eq:evaltight} \\
    a_\nu - p_{\nu} = O_p(n^{-1/2}),    \label{eq:evectight} \\
    \{ n^{1/2}[K(\rho)-K_0(\rho ;\gamma_n)], \rho \in I\} \ \text{is uniformly
      tight}. \label{eq:Ktight}
  \end{gather}
\end{lemma}

\begin{proof}
  We first remark that a matrix valued process
  $\{ X_n(\rho) \in \mR^{m \times m}, \rho \in I \}$ is uniformly
  tight iff each of the scalar processes formed from the matrix
  entries $e_k^T X_n(\rho) e_l$ is (since $m$ stays fixed throughout).
  
  We begin with \eqref{eq:Ktight} and work with the two terms in
  \eqref{eq:Wdecomp} in turn.
  Let $\mP_n, \E_n$ denote probability and expectation conditional on the event $E_{n  \delta} = \{\mu_1 \leq b_\gamma + \delta \}$. 
We show tightness of $R_n(\rho)$ on $I$
  by establishing the moment criterion of \cite[eq. (12.51)]{Bill68}:
  we exhibit $C$ such that for each
  $k,l \leq m$ and
 $\rho, \rho' \in I$,
  \begin{equation*}
    \E_n |e_k^T[R_n(\rho) -  R_n(\rho')] e_l|^2 \leq C(\rho-\rho')^2.
  \end{equation*}
  Write the quadratic form as $\bx^T \check B_n \by - \sigma_{kl} \tr \check B_n$ with
  $\bx = X_1^T e_k$ and $\by = X_1^T e_l$ being the
    $k^{th}$ and $l^{th}$ rows of $X_1$
and $\check B_n = n^{-1/2}[B_n(\rho)-B_n(\rho')]$.
  Lemma \ref{lem:basi-b26} with $p=2$ bounds the left side above by
  $C \E_n \tr \check B_n^2$.
  Now $n^{1/2} \check B_n$ has eigenvalues 
  $(\rho' - \rho) \mu_i  (\rho' - \mu_i)^{-1} (\rho - \mu_i)^{-1}$, 
  so that on $E_{n \delta}$ we
  have $\tr \check B_n^2 \leq C(\rho'-\rho)^2$, which establishes the
  moment condition.

%
   
  Tightness of $G_n(g_\rho)$, and \textit{a fortiori} that of
  $n^{-1/2} G_n(g_\rho)$, follows from that of $\widehat{M}_n(z)$ in
  \cite[Lemma 1.1]{basi04} and its following argument, by taking $x_r < b_\gamma + 3\delta$ so that $| g_\rho(z) |$ is bounded above by a constant for $z \in \mathcal{C}\cup \bar{\mathcal{C}}$.

To establish \eqref{eq:evaltight}, we work conditionally on $E_{n \delta}$. The tightness just established yields that, for given $\epsilon$, a value $M$ for which the event $E_n'$ defined by
\begin{equation*}
  \sup_{\rho\in I} n^{1/2} \| K(\rho) - K_0(\rho;\gamma_n) \| > \hf M
\end{equation*}
has $\mP_n$-probability at most $\epsilon$.
Then, we modify the argument of Section \ref{sec:eigenvalues}.
For all large enough $n$ such that $b_\gamma + 3\delta > (1 + \sqrt{\gamma_n})^2$, $\lambda_{0 \nu}(\rho) = - \rho \ssm(\rho;\gamma_n)\ell_\nu - \rho$ is the $\nu^{\text{th}}$ eigenvalue of $K_0(\rho;\gamma_n) - \rho I$ for  $\rho \in I$
(note using now $\ssm(\rho;\gamma_n)$ and $\ssf_{\gamma_n}$ in place of
$\ssm(\rho;\gamma)$ and $\ssf_{\gamma}$).
We have $\lambda_{0 \nu}(\rhonn) = 0$, and set
$\rho_{n \pm} = \rhonn \pm M n^{-1/2}$.
Since $\partial_\rho \lambda_{0 \nu}(\rho) < - 1$, we have 
$ \lambda_{0 \nu}(\rho_{n-}) \geq Mn^{-1/2}$ and
$ \lambda_{0 \nu}(\rho_{n+}) \leq -Mn^{-1/2}$.
Hence the eigenvalue perturbation bound
\eqref{eq:evalpbd} shows that on event $E_n'^{c}$,
\begin{equation*}
  \lambda_{\nu}(\rho_{n-}) \geq \hf Mn^{-1/2}, \qquad
  \lambda_{\nu}(\rho_{n+}) \leq - \hf Mn^{-1/2}.
\end{equation*}
and so $\lhat_\nu \in (\rho_{n-}, \rho_{n+} )$ as in Section \ref{sec:eigenvalues}, which implies $|\lhat_\nu- \rho_{\nu n}| \leq Mn^{-1/2}$, 
hence \eqref{eq:evaltight} is proved.

Finally, it is now easy to show \eqref{eq:evectight}. Indeed, from
the perturbation representation \eqref{eq:anu-decomp}, and noting that
$\| \mathcal{R}_{\nu n} \| \leq C$ , we have $  a_\nu - p_{\nu}  = O_p( \|D_\nu\|)$.
Since $D_\nu = K(\lhat_\nu) - K_0(\rhonn;\gamma_n)$, we have
\begin{equation*}
  \| D_\nu \|
  \leq \| K(\lhat_\nu) - K(\rhonn)\| + \|K(\rhonn)-K_0(\rhonn;\gamma_n)\|.
\end{equation*}
The first term is $O_p(n^{-1/2})$ by \eqref{eq:delK-firstorder} and
\eqref{eq:evaltight}, while the second is $O_p(n^{-1/2})$ by
\eqref{eq:Ktight}. This completes the proof.
\end{proof}

\subsection{Eigenvalue CLT}
\label{sec:eigenv-clt}


This subsection completes the proof of Theorem
\ref{th:eval-normality}. The approach is to combine the vector equations
$K(\lhat_\nu) a_\nu = \lhat_\nu a_\nu$ and
$K_0(\rhonn; \gamma_n) p_\nu = \rhonn p_\nu$ with
the expansion \eqref{eq:Kdiff2} for
$K(\lhat_\nu) - K(\rhonn)$ and the Gaussian approximation
\eqref{eq:gen(24)} for $K(\rhonn) - K_0(\rhonn; \gamma_n)$
in order to obtain the key approximate equation \eqref{eq:cltready},
namely 
\begin{equation} \label{eq:clt-ready1}
  \sqrt n (\lhat_\nu-\rhonn)[1+c(\rho_\nu)\ell_\nu + o_p(1)]
      = p_{\nu}^T R_{n \nu} p_{\nu} + o_p(1).
    \end{equation}
To ``squeeze'' the two vector equations into this scalar equation, we
use the $O_p(n^{-1/2})$ bound on $a_\nu - p_\nu$ established in the
previous subsection.

To begin, since $[K(\lhat_\nu) - \lhat_\nu I_m] a_\nu=0$, we have
\begin{equation}   \label{eq:small}
  p_{\nu}^T [K(\lhat_\nu) - \lhat_\nu I_m] p_{\nu}
  = (a_\nu-p_{\nu})^T [K(\lhat_\nu) - \lhat_\nu I_m] (a_\nu-p_{\nu})
  = O_p(n^{-1}),
\end{equation}
as $\| K(\lhat_\nu) - \lhat_\nu I_m \| = O_p(1)$ by \eqref{eq:Kfirstorder} and 
$a_\nu-p_{\nu} = O_p(n^{-1/2})$ by Lemma \ref{lem:tightness}.
Since $[ K_0(\rhonn; \gamma_n) - \rhonn I_m ] p_{\nu} = 0$, we also have
\begin{align}
  p_{\nu}^T  [K(\lhat_\nu) - \lhat_\nu I_m] p_{\nu}
  & = p_{\nu}^T [ K(\lhat_\nu) -  K_0(\rhonn ;\gamma_n) - (\lhat_\nu - \rhonn)I_m] p_{\nu} \notag \\
  & = p_{\nu}^T [ K(\lhat_\nu) -  K(\rhonn) - (\lhat_\nu - \rhonn)I_m] p_{\nu}
  + p_{\nu}^T [ K(\rhonn) -  K_0(\rhonn;\gamma_n) ] p_{\nu} \notag \\
  & = E_{n1} + E_{n2},   \label{eq:En1n2}
\end{align}
say, and will show that
\begin{align}
  \label{eq:En1}
  E_{n1} & = -(\lhat_\nu-\rhonn)[1 + c(\rho_\nu)\ell_\nu + o_p(1)], \\
  \label{eq:En2}
  E_{n2} & = n^{-1/2} p_{\nu}^T R_{n\nu} p_{\nu} + o_p(n^{-1/2}),
\end{align}
Indeed, \eqref{eq:En2} follows from \eqref{eq:Wdecomp} and Lemma
\ref{lem:tightness}. For $E_{n1}$, 
rewrite \eqref{eq:Kdiff2} as
\begin{align*}
    K(\lhat_\nu) - K(\rhonn)
   & = -(\lhat_\nu-\rhonn) n^{-1} X_1 B_{n1}(\rhonn,\rhonn) X_1^T 
      + (\lhat_\nu-\rhonn)^2 n^{-1} X_1  B_{n2}(\lhat_\nu,\rhonn)
    X_1^T.  
\end{align*}
For the first term, 
Lemma \ref{lem:quad-form-as} says that
$ n^{-1}X_1B_{n1}(\rhonn,\rhonn) X_1^T  \tooas
         c(\rho_\nu) \Sigma$, with
\begin{equation}
  c(\rho_\nu)  = \aslim n^{-1} \tr B_{n1}(\rhonn,\rhonn)
                      = \int x(\rho_\nu-x)^{-2} {\sf{F}}_\gamma(dx).
  \label{eq:crho}
\end{equation}
The second term is $O_{\rm a.s.}( (\lhat_\nu-\rho_\nu)^2 )$ by 
\eqref{eq:B_nk_quad_ineq}, \eqref{eq:B_nk_spectral_norm} and Lemma \ref{lem:quad-form-as},
and so \eqref{eq:delK} and thus the error term for $E_{n1}$ follow from Lemma \ref{lem:tightness}.


Combining \eqref{eq:small}--\eqref{eq:En2}
we obtain \eqref{eq:clt-ready1}. 
Asymptotic normality of $\sqrt n (\lhat_\nu-\rhonn)$ now follows from
Proposition \ref{prop:clt-r}, with the asymptotic variance
\begin{equation*}
  \sigma^2(\ell_\nu,\gamma)
    = [1+c(\rho_\nu) \ell_\nu]^{-2} \textrm{Var}(p_{\nu}^T R^{\nu} p_{\nu}).
\end{equation*}
Combining formula \eqref{eq:slutsky} in Appendix A for
$1+c(\rho_\nu)\ell_\nu$ with 
variance \eqref{eq:var-Rnunu}, we obtain formula \eqref{eq:sig-def}
for $\sigma^2(\ell_\nu,\gamma)$ and so complete the proof of the CLT for
$\lhat_\nu$.

\subsection{Eigenvector CLT}
\label{sec:eigenv-fluct-2}

Consider again indices $\nu$ such that $\ell_\nu > 1 + \sqrt \gamma$. 
We use \eqref{eq:main2} 
and \eqref{eq:delK} to refine decomposition (\ref{eq:Dnuto0}) to
yield \eqref{eq:Dnud0}, and consequently
$\cR_{\nu n} D_\nu p_{\nu} = n^{-1/2} \cR_{\nu n} R_{n\nu} p_{\nu} +
o_p(n^{-1/2})$
from $\cR_{\nu n} \Sigma p_{\nu} = 0$, as outlined before.
Also, from the proof of Lemma \ref{lem:tightness}, we have $\|r_\nu\| = O_p(n^{-1})$. 
We can then rewrite the perturbation decomposition
(\ref{eq:anu-decomp}) in the form (\ref{eq:anu-fluct}), namely
$\sqrt{n}(a_\nu - p_{\nu}) = - \cR_{\nu n} R_{n\nu} p_{\nu} + o_p(1)$.
The CLT for $a_\nu$ now follows from Proposition \ref{prop:clt-r}:
\begin{equation*}
  \mathcal{R}_{\nu n}R_{n\nu} p_{\nu} \toodist \mathcal{R}_\nu
  R^\nu p_{\nu} \sim N(0,\Gamma_\nu),
\end{equation*}
where $\cRn$ is given by \eqref{eq:rresolv} with $\rho_\nu$ in place
of $\rhonn$.
Observe that
\begin{equation*}
  \cRn p_\mu =
  \begin{cases}
    (\ell_\nu/\rho_\nu) (\ell_\mu-\ell_\nu)^{-1}p_\mu & \mu \neq \nu \\
    0      & \mu = \nu,
  \end{cases}
\end{equation*}
so that from \eqref{eq:Rcolcov},
\begin{equation*}
   \Gamma_\nu = 
  \cRn \textrm{Cov}(R^\nu p_{\nu}) \cRn
   = \theta_\nu \ell_\nu \cRn \Sigma \cRn + \omega_\nu \cRn K^\nu \cRn.
\end{equation*}
Using the spectral representation of $\Sigma$ and
\eqref{eq:th-om-nu} for $\theta_\nu$,
\begin{align*}
 \theta_\nu \ell_\nu \cRn \Sigma \cRn 
& = \dot \rho_\nu^{-1}
                     \sum_{\mu \neq \nu}
                     \frac{\ell_\nu \ell_\mu}{(\ell_\mu-\ell_\nu)^2} p_\mu
                     p_{\mu}^T,
\end{align*}       
while from \eqref{eq:kapmumu}, \eqref{eq:innerproduct}, and
\eqref{eq:th-om-nu} for $\omega_\nu$,
\begin{align*}   
   \omega_\nu \cRn K^\nu \cRn & = 
                     \sum_{\mu, \mu' \neq \nu}
     \frac{[\mathcal{P}^{\mu  \mu'  \nu \nu}, \kappa]}{(\ell_\mu-\ell_\nu)(\ell_{\mu'}-\ell_\nu)}
                     p_\mu p_{\mu'}^T.
\end{align*}
Thus we have $\Gamma_\nu = \sum_{\mu, \mu' \neq \nu} G_{\nu,\mu
  \mu '} \, p_\mu  p_{\mu'}^T $, with
\begin{align*}
  G_{\nu,\mu \mu '}
             & = (\ell_\nu - \ell_{\mu})^{-1}
            \{ \dot \rho_\nu^{-1} \ell_\nu \ell_\mu \delta_{\mu \mu'}
               + [\mathcal{P}^{\mu  \mu'  \nu \nu}, \kappa] \}
               (\ell_\nu - \ell_{\mu '})^{-1} \\
             & = (\mathcal{D}_\nu \tilde{\Sigma}_\nu
              \mathcal{D}_\nu)_{\mu \mu'}.
\end{align*}
Finally,  the rotation $P$ satisfies
$ p_\mu p_{\mu '}^T  = P e_{\mu} e_{\mu'} P^T$, so
we conclude that the asymptotic covariance of $\sqrt{n} a_\nu$ equals
$\Gamma = P G_\nu P^T$ while that of $\sqrt{n} P^T a_\nu$ is
$ P^T \Gamma P = G_\nu = \mathcal{D}_\nu \tilde{\Sigma}_\nu
              \mathcal{D}_\nu$.
This completes the proof of Theorem
\ref{th:evec-normality}.

\section*{Appendix A: Stieltjes transform evaluations}
\label{sec:appendixA}

Let $\ssm(z;\gamma) = \int (x-z)^{-1} \ssf_\gamma(dx)$
denote the Stieltjes transform of the companion MP law 
$\ssf_\gamma(dx)$. 
It is well defined and analytic for $z \in \mC$ outside the support of
$\ssf_\gamma$. 
In particular, it is strictly increasing for $z > b_\gamma = (1+\sqrt
\gamma)^2.$ 
In \citet{siba95} it is shown for $z \in \mC^+$ that 
$\ssm = \ssm(z;\gamma)$ is the unique solution in $\mC^+$ to the equation
\begin{equation*}
  z = - \frac{1}{\ssm} + \frac{\gamma}{1+\ssm},
\end{equation*}
or equivalently to
the quadratic equation
\begin{equation}
  \label{eq:MPE}
  z \ssm^2 + (z+1-\gamma) \ssm + 1 = 0.
\end{equation}
Allowing $z \to b_\gamma$ from above on the real axis, 
one finds that
$\ssm(b_\gamma; \gamma) = -(1+ \sqrt \gamma)^{-1}$. 

The function $d(t) = -\ssm(t;\gamma)\ell-1$, defined for $t >
b_\gamma$, is decreasing from $\ell(1+\sqrt \gamma)^{-1}-1$ at
$t=b_\gamma$ to $-1$ as $t \to \infty$.
Consequently, the equation $d(t)=0$ has a solution 
$\rho$ in $(b_\gamma,\infty)$, which is clearly unique,
exactly when $\ell > 1 + \sqrt \gamma$, i.e. above
the phase transition.
The equation may be rewritten as
\begin{equation}
  \label{eq:rho-sol}
  \int(\rho-x)^{-1} \ssf_\gamma(dx) = \ell^{-1},
\end{equation}
and the solution $\rho$ can be found by inserting $\ssm =  - 1/\ell$
into \eqref{eq:MPE} and solving for $z = \rho$, which yields
\begin{equation}
   \label{eq:rhosol}
  \rho = \rho(\ell,\gamma) = \ell + \ell\gamma/(\ell-1).
\end{equation}

Dropping the explicit dependence on $\gamma$ in $\ssm(z)$, we
differentiate with respect to $\ell$ in the equation 
$\ssm(\rho) = -\ell^{-1}$ and in (\ref{eq:rhosol}) to obtain
\begin{equation*}
  \ssm'(\rho) \rho'(\ell)   = \ell^{-2}, \qquad \quad
  \rho'(\ell)  = 1 - \gamma/(\ell-1)^2.
\end{equation*}
From the Stieltjes transform definition we then obtain
\begin{equation*}
  \int (\rho-x)^{-2} \ssf_\gamma(dx) 
    = \ssm'(\rho)
    = [\ell^2 \rho'(\ell)]^{-1}.
\end{equation*}

We can now evaluate the (almost sure) limits
\begin{align}
    \theta_\nu & = \lim n^{-1} \tr B_n^2(\rhonn) 
               = \int \rho_\nu^2 (\rho_\nu-x)^{-2} \ssf_\gamma(dx)
                 \notag \\
               & = \frac{1}{\dot \rho_\nu} \bigg(
                 \frac{\rho_\nu}{\ell_\nu}\bigg)^2 
    = \frac{(\ell_\nu-1+\gamma)^2}{(\ell_\nu-1)^2-\gamma} \label{eq:thetanu} \\
  c(\rho_\nu)
  & = \lim n^{-1} \tr B_{n1}(\rhonn,\rhonn)
    = \int x (\rho_\nu-x)^{-2} \ssf_\gamma(dx) \notag \\
  & = \int [\rho_\nu(\rho_\nu-x)^{-2} -(\rho_\nu-x)^{-1}]
    \ssf_\gamma(dx) \notag  \\
  & = \frac{\rho_\nu}{\ell_\nu^2 \dot \rho_\nu} - \frac{1}{\ell_\nu} 
    = - \frac{1}{\dot \rho_\nu} \frac{\partial}{\partial \ell_\nu}
    \left( \frac{\rho_\nu}{\ell_\nu} \right)
    = \frac{\gamma}{(\ell_\nu-1)^2 -\gamma},  \label{eq:crhonu}
\end{align}
and, recalling \eqref{eq:rho-sol}, we also have
\begin{equation}
  \label{eq:slutsky}
  \begin{split}
  1 + \ell_\nu c(\rho_\nu)
   & = \ell_\nu \left( c(\rho_\nu) + \int (\rho_\nu-x)^{-1}
     \ssf_\gamma(dx) \right) \\
   & = \ell_\nu \rho_\nu \int (\rho_\nu-x)^{-2} \ssf_\gamma(dx)
     = \frac{1}{\dot \rho_\nu}  \left( \frac{\rho_\nu}{\ell_\nu} \right)
     = \frac{1+\gamma(\ell_\nu-1)^{-1}}{1-\gamma(\ell_\nu-1)^{-2}}.
  \end{split}
\end{equation}



\begin{lemma}
  \label{lem:omegaval} With $B_n$ defined at (\ref{eq:Bn-def}) and $ t = \rho_{\nu n}$,
  \begin{equation*}
    \omega_\nu = \plim n^{-1} \sum_{i=1}^n b_{n,ii}^2
           = (\rho_\nu/\ell_\nu)^2
           = [1+\gamma/(\ell_\nu -1)]^2.
  \end{equation*}
\end{lemma}
\begin{proof}
  The argument is a  variant of Lemma 6.1 of \cite{bayo08}.
  Denote the $i$th column of $X_2$ by $x_i$.
The two forms for $B_n$ in (\ref{eq:Bn-def}) yield for $b_{n,ii}$
respectively 
\begin{equation}
  b_{n,ii} 
      = 1 + n^{-1}x_i^T(\rhonn-S_n)^{-1}x_i 
      = \rhonn [(\rhonn - C_n)^{-1}]_{ii}, \label{eq:bnii}
\end{equation}
where $S_n = n^{-1} X_2X_2^T$ and $C_n = n^{-1} X_2^TX_2$.
Both forms are useful; we first rewrite the first form after setting 
$X_{2i}$ for  $X_2$ with the $i$th column deleted, and $S_{ni} := n^{-1} X_{2i} X_{2i}^T$. 
Then $S_n = S_{ni} + n^{-1} x_i x_i^T$ and
\begin{align*}
  b_{n,ii} 
           &  = 1 + n^{-1} x_i^T (\rho_\nu - S_{ni}
           -n^{-1} x_i x_i^T)^{-1} x_i 
              = (1-c_{n,ii})^{-1}, 
\intertext{where we have used the Woodbury formula $1+u^T(A-uu^T)^{-1}u =
(1-u^TA^{-1}u)^{-1}$,}
  c_{n,ii} & = n^{-1} x_i^T (\rho_\nu - S_{ni})^{-1} x_i.
\end{align*}
The event $E_{n\delta} = \{\mu_1 \leq  b_\gamma+\delta \} = \{ \|S_n\|, \|S_{ni} \|  \leq b_\gamma+\delta, i= 1, \ldots, n \}$
(since $\| S_{ni} \| \leq \|S_n\|$)
is symmetric with respect to permutations of the columns of $X_2$ and
is of high probability: $\mP(E_{n\delta})\to 1$. 

Apply Lemma \ref{lem:quad-form-as} 
with $X_1 = x_i$ and 
$B_n = (\rhonn - S_{ni})^{-1}$. 
We have
\begin{align*}
  \aslim c_{n,ii}
    & = \aslim n^{-1} \tr B_n(\rhonn) 
  = \gamma \int (\rho_\nu-x)^{-1} F_\gamma(dx)  \\
    & = -\ssm(\rho_\nu) - \frac{1-\gamma}{\rho_\nu}
      = 1 + \frac{1}{\rho_\nu \ssm(\rho_\nu)}.
\end{align*}
The second equality follows from the analog of \eqref{eq:as-fnal} for
$F_n$ and $F_\gamma$, in the same manner as for \eqref{eq:lln}.
The third equality uses the
companion relation \eqref{eq:companion}, and
the final equality rewrites the Marchenko-Pastur equation \eqref{eq:MPE}.

Thus $\plim b_{n,ii} = - \rho_\nu \ssm(\rho_\nu) = \rho_\nu/\ell_\nu$, and
$\plim b_{n,ii}^2 = \rho_\nu^2/\ell_\nu^2 = b^2$, say.
Since the event $E_{n\delta}$ is invariant to permutation of columns
of $X_2$, writing $\E_n$ for expectation conditional on $E_{n\delta}$ we have
\begin{equation*}
  \E_n \Big|n^{-1}\sum_{i=1}^nb_{n,ii}^2-b^2\Big|
  \leq n^{-1} \sum_{i=1}^n \E_n |b_{n,ii}^2-b^2|
  = \E_n |b_{n,11}^2-b^2| \to 0
\end{equation*}
because $|b_{n,11}| = |\rhonn [(\rhonn-C_n)^{-1}]_{11}|
\leq C(\rho_\nu,\delta)$ on event $E_{n\delta}$.
Thus $\plim n^{-1} \sum_i b_{n,ii}^2 = b^2$ and the proof is done.
\end{proof}



\section*{Appendix B: Bai-Yao CLT for real valued data}
\label{sec:appendix-b:-bai}

We verify that the martingale method of Baik-Silverstein, [BS] below,
presented in the Appendix of \cite{cdf09}, extends to
establish the Bai-Yao CLT, 
for vectors of random symmetric bilinear forms,
Theorem 7.1 of \cite{bayo08},
under a spectral norm bound condition. We focus on the version for
real-valued data.

\bigskip

Consider a sequence of zero-mean vectors 
$(x_i,y_i) $ i.i.d. as $(x,y) \in \mR^L \times \mR^L$ with 
\begin{equation*}
  \text{Cov}
  \begin{pmatrix}
    x \\ y
  \end{pmatrix}
  = \Gamma = 
  \begin{pmatrix}
    \Gamma^{xx} & \Gamma^{xy} \\
    \Gamma^{yx} & \Gamma^{yy} 
  \end{pmatrix},
  \qquad \quad
  \rho_\ell = \Gamma^{xy}_{\ell \ell} = \E \,x_{\ell} y_{\ell}.
\end{equation*}
Let $J = \Gamma^{xx} \circ \Gamma^{yy} + \Gamma^{xy} \circ
\Gamma^{yx}$, that is
\begin{equation} \label{eq:Jdef}
  J_{\ell \ell'} = \E (x_{\ell} x_{\ell'}) \E (y_{\ell}y_{\ell'})
                 + \E (x_{\ell}y_{\ell'}) \E (y_{\ell}x_{\ell'}),
\end{equation}
and let $K = (K_{\ell \ell'})$ be the  partial
cumulant matrix 
\begin{equation} \label{eq:Kellpdef}
  K_{\ell \ell'} = \E( x_{\ell} y_{\ell} x_{\ell'} y_{\ell'}) 
       - \E (x_{\ell} y_{\ell}) \E (x_{\ell'}y_{\ell'})
       - \E (x_{\ell} y_{\ell'}) \E (x_{\ell'}y_{\ell})
       - \E (x_{\ell} x_{\ell'}) \E (y_{\ell}y_{\ell'}).
     \end{equation}

We introduce a notation for rows $X_\ell^T$ and $Y_\ell^T$ of data
matrices based on $n$ observations
\begin{equation*}
  (x_{\ell i})_{L \times n} =
  \begin{bmatrix}
    X_1^T \\ \vdots \\ X_L^T
  \end{bmatrix}, \qquad
  (y_{\ell i})_{L \times n} =
  \begin{bmatrix}
    Y_1^T \\ \vdots \\ Y_L^T
  \end{bmatrix},
\end{equation*}

\begin{theorem}
\label{th:bybs}
  Let $B_n = (b_{n,ij})$ be random symmetric $n \times n$ matrices,
  independent of $\{(x_i,y_i), i \in \mN\}$, such that $\|B_n\| \leq
  a$ for all $n$, and
  \begin{equation*}
    \omega  = \plim n^{-1} \sum_{i=1}^n b_{n,ii}^2, \qquad
    \theta  = \plim n^{-1} \tr B_n^2.
  \end{equation*}
Let $Z_n \in \mR^L$ have components
\begin{equation*}
  Z_{n,\ell} = n^{-1/2} [X_\ell^T B_n Y_\ell - \rho_\ell \tr B_n].
\end{equation*}
Then 
\begin{equation} \label{eq:eqD}
  Z_n \stackrel{\mathcal{D}}{\to} Z \sim N_L(0,D), \quad  \text{with}
  \quad
  D = \theta  J + \omega  K.
\end{equation}
\end{theorem}

\textit{Remark.} Alexei Onatski (personal communication) has noted
that an extra condition, such as the norm bound above, is needed in
the Bai-Yao result. Indeed, let 
$e_n$ denote the vector of ones normalized to unit length, and set
$B_n = \sqrt{n} e_n e_n^T$. Then $\omega = 0$ and $\theta =
1$.  However, with $L=1$ and $x_i=y_i$ i.i.d. standard normal, we have 
$Z_n = (n^{-1/2}\sum_1^n x_i)^2-1 \equiv \chi_1^2 - 1$ for all $n$.
Note that  $\| B_n \| = \sqrt n$.
Inspection of the proof below suggests that the result continues to hold
if $\|B_n\| = O_p(n^{1/4-\epsilon})$.

\bigskip
\textit{Details.} \ 
We use the Cramer-Wold device and show for each $c \in \mR^L$ that
$c^T Z_n \stackrel{\mathcal{D}}{\to} N(0,c^TDc)$.
The modifications to the Baik-Silverstein argument to deal with linear
combinations $c^T Z_n$ and symmetric bilinear forms are essentially
just notational, but they are nevertheless set out below.

Start with a single bilinear form $X^TBY = \sum_{i,j=1}^n x_i b_{ij} y_j$
built from i.i.d. zero mean vectors $(x_i,y_i) \in \mR^2$ with
\begin{equation*}
  \text{Cov}
  \begin{pmatrix}
    x_1 \\ y_1
  \end{pmatrix}
  =
  \begin{pmatrix}
    \gamma^{xx} & \rho \\
    \rho        & \gamma^{yy}
  \end{pmatrix},
  \qquad \quad 
  \kappa = \E (x_1y_1-\rho)^2,
\end{equation*}
so that the centering term
$\E X^TBY = \rho \tr B.$
The symmetry of $B$ allows a decomposition
\begin{equation}   \label{eq:single-decomp}
  X^T B Y - \rho \tr B 
     = \sum_{i=1}^n (x_i y_i - \rho)b_{ii} + x_i S_i(y) + y_i S_i(x),
\end{equation}
where we set
\begin{equation*}
  S_i(y) = \sum_{j=1}^{i-1} b_{ij} y_j.
\end{equation*}

\begin{lemma}
  \label{lem:bounds}
(a) If $B$ is a symmetric matrix, then
\begin{equation}  \label{eq:varbd}
  \E [X^T B Y - \rho \tr B]^2 
     \leq C \tr B^2,
\end{equation}
where $C = 3 \max(\kappa,\gamma^{xx}\gamma^{yy})$.

(b) If $B = B_n$ is a random symmetric matrix independent of $\{(x_i,y_i),
i \in \mN\}$, and $\E \tr (B_n^2) = o(n^{2})$, then
\begin{equation*}
  \plim n^{-1} X^T B_n Y 
    = \rho \plim n^{-1} \tr B_n.
\end{equation*}
If event $E_n$ is $B_n$-measurable and $\mP(E_n) \to 1$, then the
condition $\E[\tr(B_n^2),E_n] = o(n^2)$ suffices. 
\end{lemma}

\begin{proof}
  (a). Using the decomposition \eqref{eq:single-decomp}, the left side
  of \eqref{eq:varbd} is bounded by
  \begin{equation}
    \label{eq:first}
    3 \sum_i \E (x_i y_i - \rho)^2 b_{ii}^2 
       + \E x_i^2 S_i^2(y) + \E y_i^2 S_i^2(x).
  \end{equation}
Using independence of $y_i$, we calculate
$\E S_i^2(y) = \gamma^{yy} \sum_{j=1}^{i-1} b_{ij}^2$.
and then by independence of $x_i$ from $S_i(y)$ we obtain
\begin{equation*}
  \sum_i \E x_i^2 S_i^2(y) = \gamma^{xx} \gamma^{yy} \sum_{i>j}
  b_{ij}^2.
\end{equation*}
The third term of \eqref{eq:first} is handled similarly, so that
\eqref{eq:first} is bounded by
\begin{equation*}
  3 \kappa \sum_i b_{ii}^2 + 3 \gamma^{xx} \gamma^{yy} \sum_{i\neq j}
  b_{ij}^2 
     \leq  C \tr (B^2).
\end{equation*}

(b) Conditioning on $B_n$ is harmless due to the assumed independence,
so part (a) yields
\begin{equation*}
  \E [n^{-1} X^T B_n Y - \rho n^{-1} \tr B_n]^2
      \leq C n^{-2} \E \tr B_n^2 
      \to 0,
\end{equation*}
which suffices for our conclusion.
\end{proof}

\begin{lemma}
  \label{lem:cge}
In this setting, i.e. with $X = (x_i), Y =(y_i), \rho = \E x_1 y_1$
and $B = B_n$ as in Theorem \ref{th:bybs}, we have
\begin{align}
  T_n(Y;B) 
    & = n^{-1} \sum_{i=1}^n b_{ii} S_i(y) \stackrel{p}{\to} 0  
        \label{eq:T-lim} \\
  Q_n(X,Y;B) 
    & = n^{-1} \sum_{i=1}^n S_i(x) S_i(y) \stackrel{p}{\to} 
        \hf \rho (\theta-\omega). \label{eq:Q-lim}
\end{align}
\end{lemma}
\begin{proof}
To establish \eqref{eq:T-lim}, begin by writing
\begin{equation*}
  S_i(y) = \sum_{j=1}^n \check b_{ij} y_j = (B_LY)_i,
\end{equation*}
where the lower triangular matrix $B_L = (\check b_{ij})$ is defined
by $\check b_{ij} = b_{ij}$ if $i>j$ and $0$ otherwise.
We have
\begin{equation*}
  \E [S_i(y) S_{i'}(y) | B]
   = \sum_{j,j'=1}^n \check b_{ij} \check b_{i'j'} \E y_j y_{j'}
   = \sigma_y^2 \sum_{j=1}^n \check b_{ij} \check b_{i'j}
   = \sigma_y^2 (B_L B_L^T)_{ii'},
\end{equation*}
where $\sigma_y^2 = \E y_1^2$. Consequently,
\begin{equation*}
  \E [T_n^2(Y,B) | B]
    = \sigma_y^2 n^{-2} \sum_{i,i'=1}^n b_{ii} b_{i'i'} (B_L B_L^T)_{ii'}.
\end{equation*}
Let $b = (b_{ii}) \in \mR^n$. The double sum may be bounded via
\begin{equation*}
  b^T B_L B_L^T b 
    \leq n \| b \|_\infty^2 \| B_L B_L^T \|
    \leq (n \log^2 n) c^2 a^4,
\end{equation*}
where we used the bound of \cite{math93} $\| B_L \| \leq c \log n \|B \|$, $ \| b \|_{\infty} \leq \| B \|$
and the operator norm bound on $B$. 
Hence $\E T_n^2(Y;B) = O(n^{-1} \log^2 n)$ and so \eqref{eq:T-lim}
follows. 

To establish \eqref{eq:Q-lim}, write 
$Q_n(X,Y;B) = n^{-1} X^TB_L^TB_LY$
and apply Lemma \ref{lem:bounds}(b) with $\check{B}_n =B_L^TB_L$.
We have
\begin{equation*}
 \tr B_L^TB_L = \sum_{i>j} b_{ij}^2 = \hf \left[ \tr B^2 - \sum_i
    b_{ii}^2 \right], 
\end{equation*}
and since $\tr (B_L^TB_L)^2 \leq n \|B_L^T B_L\|^2 \leq c^4 a^4n\log^4n$,
Lemma \ref{lem:bounds}(b) implies that
\begin{equation*}
  \plim Q_n =  \plim \rho \, n^{-1}\tr B_L^T B_L
      = \hf \rho(\theta-\omega).  \qedhere
\end{equation*}
\end{proof}


Now apply all this in the setting of the theorem. 
Let $\mathcal{F}_{n,i}$ be the $\sigma$-field generated by 
$B_n$ and 
$\{ (x_j,y_j), 1 \leq j \leq i\}$, and
$\E_{i-1}$ denote conditional expectation
w.r.t. $\mathcal{F}_{n,i-1}$. 
We use the decomposition \eqref{eq:single-decomp} to obtain
\begin{align}
  c^T Z_n
    & = n^{-1/2} \sum_\ell c_\ell [X_\ell^T B_n Y_\ell - \rho_\ell \tr
      B_n ] \notag \\
    & = \sum_{i=1}^n Z_{di} + Z_{yi} + Z_{xi} 
      = \sum_{i=1}^n Z_{ni},  \label{eq:decomp}
\end{align}
where for $a \in \{d,y,x\}$ we define martingale differences
\begin{equation*}
  Z_{ai} = n^{-1/2} \sum_\ell c_\ell w_{a \ell i} v_{a \ell i},
\end{equation*}
and the terms $v_{a \ell i} \in \mathcal{F}_{n,i-1}$ and
$w_{a \ell i}$ are given by
\begin{equation*}
\begin{array}[h]{ccc}
\qquad  a \qquad  & \qquad  w_{a \ell i} \qquad & 
     \qquad v_{a \ell i}
     \qquad  \\
  \hline
  d & x_{\ell i} y_{\ell i} - \rho_\ell & b_{ii}  \\
  y & x_{\ell i} & S_i(y_\ell) \\                                         
  x & y_{\ell i} & S_i(x_\ell) \\                                         
\end{array}
\end{equation*}

We apply the martingale CLT as in [BS], so we need to check a
Lindeberg condition and show that the quadratic variation converges. 
We do the second step first. 
The limiting variance in the central limit theorem may be found
from $v^2 = \plim V_n^2$, where 
\begin{equation}  \label{eq:vdecomp}
  V_n^2 = \sum_{i=1}^n \E_{i-1} [Z_{ni}^2]
        = V_{n,dd} + 2(V_{n,dy} + V_{n,dx}) + V_{n,yy} + V_{n,xx} 
           +2 V_{n,xy},
\end{equation}
where
we have set
\begin{equation*}
  V_{n,ab} = \sum_{i=1}^n \E_{i-1} [Z_{ai} Z_{bi}]
        = \sum_{\ell,\ell'} c_\ell c_{\ell'} \E[w_{a \ell} w_{b
    \ell'}] \cdot
    n^{-1} \sum_{i=1}^n v_{a \ell i} v_{b \ell' i},
\end{equation*}
since the distribution of $w_{a \ell i}$ is independent of $i$.
The values of the two terms in the product as $a, b$ vary in
$\{d,y,x\}$ are given in the following table
\begin{equation*}
  \begin{array}[h]{ccccc}
    & \E[w_{a \ell} w_{b \ell'}] & n^{-1} \sum_{i=1}^n v_{a \ell i}
                                   v_{b \ell' i} & & \\
    \hline
    a,b = d & M_{\ell \ell'} & n^{-1} \sum_{i=1}^n b_{ii}^2 &
                                                              \stackrel{p}{\to}
                                                   & \omega \\
    a=d,\ \ b\in \{x,y\} & M^{\check b}_{\ell \ell'}
                        & T_n(X_\ell^b,B) & \stackrel{p}{\to}
                                                 & 0 \\
    a, b\in \{x,y\} & \Gamma^{\check a \check b}_{\ell \ell'}
        & Q_n(X_\ell^a,X_{\ell'}^b,B) & \stackrel{p}{\to} & \hf (\theta-\omega)
                                              \Gamma^{ab}_{\ell \ell'}
  \end{array}
\end{equation*}
In this table, we are using the notations
\begin{align*}
  M_{\ell \ell'} & = \E (x_{\ell} y_{\ell} - \rho_{\ell}) (x_{\ell'}
                   y_{\ell'} - \rho_{\ell'}) \\
  M^x_{\ell \ell'} & = \E (x_{\ell} y_{\ell} - \rho_{\ell}) x_{\ell'},
                   \quad
  M^y_{\ell \ell'} = \E (x_{\ell} y_{\ell} - \rho_{\ell}) y_{\ell'},         
\end{align*}
along with the convention that if $a \in \{x,y\}$, then $\check a$
denotes the complementary element, and also writing
\begin{equation*}
  X^a_\ell =
  \begin{cases}
    (x_{\ell i})_{i=1}^n  & \text{if } a = x \\
      (y_{\ell i})_{i=1}^n  & \text{if } a = y.
  \end{cases}
\end{equation*}
The convergence in probability statements in the right hand column
follow from Lemma \ref{lem:cge}.

Consequently, combining terms according to 
\eqref{eq:vdecomp} and the previous limits, we finally get
\begin{align*}
  v^2 & = \omega c^T M c + 0 + (\theta-\omega) c^T (\Gamma^{xx} \hprod
        \Gamma^{yy} +  \Gamma^{xy} \hprod \Gamma^{yx}) c \\
        & = \omega c'(M-J)c + \theta c'Jc
          = \omega c'K c + \theta c'Jc
\end{align*}
after noting that $K = M  - J$.

The verification of the Lindeberg condition is straightforward.
[BS] show a closure property: for random variables $X_1,X_2$ and
positive $\epsilon$,
\begin{equation*}
  \E(|X_1+X_2|^2 \mathbbm{1}( |X_1+X_2|\geq\epsilon )
    \leq 4 (\E(|X_1|^2 \mathbbm{1}( |X_1|\geq \epsilon/2 ) ) + 
            \E(|X_2|^2 \mathbbm{1}( |X_2|\geq \epsilon/2 ) )).
\end{equation*}
From this property, it suffices to establish the Lindeberg condition
for individual martingale difference sequences
\begin{equation*}
  Z_{\ell i} 
   = (x_{\ell i} y_{\ell i} - \rho_\ell) b_{ii}, \quad
     x_{\ell i} S_i(y_\ell), \quad \text{and} \quad 
     y_{\ell i} S_i(x_\ell).
\end{equation*}
This follows just as in [BS] with minor changes to allow for pairs
$(x_{\ell i},y_{\ell i})$ in place of $(x_i,x_i)$.

\section*{Appendix C: A perturbation Lemma }
\label{sec:appendix-c:-proof}

Several variants of this
lemma appear in the literature,
most based on the approach of \cite{Kato1980}. The version here is
modified from \cite{Paul2005}. 
Let the
eigenvalues of a symmetric matrix $A$ be denoted by $\lambda_1(A) \geq \cdots
\geq \lambda_m(A)$, with the convention that $\lambda_0(A) = \infty$ and
$\lambda_{m+1}(A) = - \infty$. 
Let $P_s$ denote the projection matrix onto the possibly multi-dimensional
eigenspace corresponding to $\lambda_s(A)$ and define
\begin{displaymath}
H_r(A) = \sum_{s \neq r} \frac{1}{\lambda_s(A) -
\lambda_r(A)} P_{s}(A).  
\end{displaymath}
Observe that $H_r(A) (A - \lambda_r(A)I_m) = I_m - P_r(A)$; in this sense we may
say that $H_r(A)$ is the resolvent of $A$ ``evaluated at $\lambda_r$.''

\begin{lemma} \label{lemma:evec_perturb_bound}
Let $A$ and $B$ be symmetric $m \times m$ matrices. 
Suppose that $\lambda_r(A)$ is a simple eigenvalue of $A$ with
\begin{displaymath}
  \delta_r(A) = \min \{ |\lambda_j(A) - \lambda_r(A)| : 1 \leq j \neq
  r \leq m \}.
\end{displaymath}
If $\|B\| < \delta_r(A)/3$, then $\lambda_r(A+B)$ is also simple.
Let $\mathbf{p}_r(A)$ and $\mathbf{p}_r(A+B)$ denote the unit eigenvectors
associated with $\lambda_r(A)$ and $\lambda_r(A+B)$, with their signs chosen
so that $\gamma_r = \mathbf{p}_r(A)^T\mathbf{p}_r(A+B) \geq 0$. 
Then
\begin{equation}\label{eq:eig_perturb_first}
\mathbf{p}_r(A+B) - \mathbf{p}_r(A) = - H_r(A) B
\mathbf{p}_r(A) + R_r,
\end{equation}
with the quadratic error bound
\begin{equation}
  \label{eq:abs-err-bd}
  \| R_r \| \leq 10 \delta_r^{-2}(A) \| B \|^2.
\end{equation}
\end{lemma}

\begin{proof}
We first note that if $\| B \| < \delta_r(A)/3$, then
 the usual eigenvalue perturbation bound $|\lambda_j(A+B)
- \lambda_j(A)| \leq \|B\|$ ensures that
\begin{equation*}
  \min_{j \neq r} |\lambda_j(A+B) - \lambda_r(A+B)| > \delta_r(A)/3,
\end{equation*}
and so in particular $\lambda_j(A+B)$ is simple.

  So, let $\mathbf{p}_r(A+B)$ be the unit eigenvector associated with
$\lambda_r(A+B)$, and rewrite the defining expression
$(A+B) \mathbf{p}_r(A+B) = \lambda_r(A+B) \mathbf{p}_r(A+B)$ as
\begin{displaymath}
  (A - \lambda_r(A) I_m) \mathbf{p}_r(A+B)
   = - B \mathbf{p}_r(A+B) + \{ \lambda_r(A+B) - \lambda_r(A) \}
   \mathbf{p}_r(A+B). 
\end{displaymath}
Write $\Delta \lambda_r$ for $\lambda_r(A+B) - \lambda_r(A)$ and 
$P_r^\perp$ for the orthogonal projector $I_m-P_r(A)$. 
Premultiply the previous display by $H_r(A)$; since
$H_r(A) (A - \lambda_r(A) I_m) = P_r^\perp$, we get
\begin{align}
  \label{eq:proj1}
P_r^\perp \mathbf{p}_r(A+B)
   & = - H_r(A) B \mathbf{p}_r(A+B) + \Delta \lambda_r H_r(A)
   \mathbf{p}_r(A+B) \\
   & = - H_r(A) B \mathbf{p}_r(A) - H_r(A) B \Delta \mathbf{p}_r 
       + \Delta \lambda_r H_r(A)   \mathbf{p}_r(A+B),
      \label{eq:proj2}
\end{align}
where $\Delta \mathbf{p}_r = \mathbf{p}_r(A+B) - \mathbf{p}_r(A)$.

Using \eqref{eq:proj1} along with $| \Delta \lambda_r | \leq \| B \|$
and $\| H_r(A) \| \leq \delta_r^{-1}(A)$, we get an initial bound
which is first order in $\| B \|$:
\begin{equation}
  \label{eq:first-order}
  y_r := \| P_r^\perp \mathbf{p}_r(A+B) \| 
    \leq 2 \delta_r^{-1} (A) \| B \|.
\end{equation}

Turning to the main focus of interest, clearly
\begin{equation}
  \label{eq:pr-decomp}
  \mathbf{p}_r(A+B) 
    = P_r^\perp \mathbf{p}_r(A+B)  + \gamma_r \mathbf{p}_r(A), 
      \qquad \gamma_r = \mathbf{p}_r(A)^T\mathbf{p}_r(A+B),
\end{equation}
so that from \eqref{eq:proj2} we arrive at the decomposition
\eqref{eq:eig_perturb_first}:
\begin{align*}
  \mathbf{p}_r(A+B) - \mathbf{p}_r(A) 
    & = - H_r(A) B \mathbf{p}_r(A) - R_1 + R_2 - R_3.
\end{align*}
with
\begin{displaymath}
  R_1 = H_r(A) B \Delta \mathbf{p}_r, \qquad
  R_2 = \Delta \lambda_r H_r(A)   \mathbf{p}_r(A+B), \qquad 
  R_3 = (1-\gamma_r) \mathbf{p}_r(A)
\end{displaymath}


Let us bound these error terms, starting with $R_2$.
Since $H_r(A) = H_r(A) P_r^\perp$, we may rewrite 
$R_2$ as $\Delta \lambda_r H_r(A) P_r^\perp \mathbf{p}_r(A+B)$, so
that
\begin{displaymath}
  \| R_2 \| 
   \leq | \Delta \lambda_r| \| H_r (A) \| \| P_r^\perp
         \mathbf{p}_r(A+B) \|  
   \leq \delta_r^{-1}(A) \| B \| y_r.
\end{displaymath}

From the definitions $\gamma_r^2 + y_r^2 = 1$,
so that  $  1 - \gamma_r = 1 - (1 - y_r^2)^{1/2} \leq y_r^2$,
and so, from \eqref{eq:first-order},
\begin{displaymath}
  \| R_3 \| = 1- \gamma_r \leq 2 \delta_r^{-1}(A) \| B \| y_r.
\end{displaymath}

Using \eqref{eq:pr-decomp},
\begin{displaymath}
  \Delta \mathbf{p}_r = P_r^\perp \mathbf{p}_r(A+B) + (\gamma_r - 1)
  \mathbf{p}_r(A),
\end{displaymath}
and so, from $1 - \gamma_r \leq y_r^2$
and $y_r \leq 1$, we have
\begin{equation}
  \label{eq:delpr}
  \| \Delta \mathbf{p}_r \| \leq y_r + y_r^2 \leq 2 y_r,
\end{equation}
and so
\begin{displaymath}
  \| R_1 \| 
   \leq \| H_r(A) B \Delta \mathbf{p}_r \|   
   \leq   2 \delta_r^{-1}(A) \| B \| y_r.
\end{displaymath}
Combining the bounds for $R_1, R_2$ and $R_3$, we obtain
\begin{equation}
  \label{eq:Rbd}
  \| R_r \| \leq 5 \delta_r^{-1} (A) \| B \| y_r,
\end{equation}
and combined with
\eqref{eq:first-order}, we recover \eqref{eq:abs-err-bd}. 
\end{proof}

\textit{Acknowledgement:} \   
We thank Matt McKay and David Morales for their discussion of, and many comments 
on, this manuscript.
This work supported in part by NIH R01 EB001988
(IMJ, JY) and by a Samsung Scholarship (JY).

\bibliographystyle{plainnat}
\bibliography{paul5,baiyao}

\begin{thebibliography}{26}
\providecommand{\natexlab}[1]{#1}
\providecommand{\url}[1]{\texttt{#1}}
\expandafter\ifx\csname urlstyle\endcsname\relax
  \providecommand{\doi}[1]{doi: #1}\else
  \providecommand{\doi}{doi: \begingroup \urlstyle{rm}\Url}\fi

\bibitem[Bai and Silverstein(2009)]{basi09}
Z.~D. Bai and Jack Silverstein.
\newblock \emph{Spectral Analysis of Large Dimensional Random Matrices}.
\newblock Springer, New York, second edition, 2009.

\bibitem[Bai and Silverstein(2004)]{basi04}
Z.~D. Bai and Jack~W. Silverstein.
\newblock C{LT} for linear spectral statistics of large-dimensional sample
  covariance matrices.
\newblock \emph{Ann. Probab.}, 32\penalty0 (1A):\penalty0 553--605, 2004.

\bibitem[Bai and Yin(1993)]{bayi93}
Z.~D. Bai and Y.~Q. Yin.
\newblock Limit of the smallest eigenvalue of a large-dimensional sample
  covariance matrix.
\newblock \emph{Ann. Probab.}, 21\penalty0 (3):\penalty0 1275--1294, 1993.

\bibitem[Bai and Ding(2012)]{badi12}
Zhidong Bai and Xue Ding.
\newblock Estimation of spiked eigenvalues in spiked models.
\newblock \emph{Random Matrices: Theory and Applications}, 01\penalty0
  (02):\penalty0 1150011, 2012.

\bibitem[Bai and Yao(2008)]{bayo08}
Zhidong Bai and Jian-feng Yao.
\newblock Central limit theorems for eigenvalues in a spiked population model.
\newblock \emph{Ann. Inst. Henri Poincar\'e Probab. Stat.}, 44\penalty0
  (3):\penalty0 447--474, 2008.

\bibitem[Bai and Yao(2012)]{bayo12}
Zhidong Bai and Jianfeng Yao.
\newblock On sample eigenvalues in a generalized spiked population model.
\newblock \emph{J. Multivariate Anal.}, 106:\penalty0 167--177, 2012.

\bibitem[Baik and Silverstein(2006)]{basi06}
Jinho Baik and Jack~W. Silverstein.
\newblock Eigenvalues of large sample covariance matrices of spiked population
  models.
\newblock \emph{Journal of Multivariate Analysis}, 97:\penalty0 1382--1408,
  2006.

\bibitem[Benaych-Georges and Nadakuditi(2011)]{bgrn11}
Florent Benaych-Georges and Raj~Rao Nadakuditi.
\newblock The eigenvalues and eigenvectors of finite, low rank perturbations of
  large random matrices.
\newblock \emph{Advances in Mathematics}, 227\penalty0 (1):\penalty0 494 --
  521, 2011.
\newblock ISSN 0001-8708.

\bibitem[Benaych-Georges et~al.(2011)Benaych-Georges, Guionnet, and
  Maida]{bggm11}
Florent Benaych-Georges, Alice Guionnet, and Mylène Maida.
\newblock Fluctuations of the extreme eigenvalues of finite rank deformations
  of random matrices.
\newblock \emph{Electron. J. Probab.}, 16:\penalty0 1621--1662, 2011.

\bibitem[Billingsley(1968)]{Bill68}
Patrick Billingsley.
\newblock \emph{Convergence of probability measures}.
\newblock John Wiley \& Sons, Inc., New York-London-Sydney, 1968.

\bibitem[Bloemendal et~al.(2016)Bloemendal, Knowles, Yau, and Yin]{bkyy16}
Alex Bloemendal, Antti Knowles, Horng-Tzer Yau, and Jun Yin.
\newblock On the principal components of sample covariance matrices.
\newblock \emph{Probability Theory and Related Fields}, 164\penalty0
  (1):\penalty0 459--552, Feb 2016.

\bibitem[Capitaine et~al.(2009)Capitaine, Donati-Martin, and F\'eral]{cdf09}
Mireille Capitaine, Catherine Donati-Martin, and Delphine F\'eral.
\newblock The largest eigenvalues of finite rank deformation of large {W}igner
  matrices: convergence and nonuniversality of the fluctuations.
\newblock \emph{Ann. Probab.}, 37\penalty0 (1):\penalty0 1--47, 2009.

\bibitem[Couillet and Hachem(2013)]{coha13}
R.~Couillet and W.~Hachem.
\newblock Fluctuations of spiked random matrix models and failure diagnosis in
  sensor networks.
\newblock \emph{IEEE Transactions on Information Theory}, 59\penalty0
  (1):\penalty0 509--525, 2013.

\bibitem[Ding(2015)]{ding15}
Xue Ding.
\newblock Convergence of sample eigenvectors of spiked population model.
\newblock \emph{Communications in Statistics - Theory and Methods}, 44\penalty0
  (18):\penalty0 3825--3840, 2015.

\bibitem[Fan et~al.(2018)Fan, Johnstone, and Sun]{fjs18}
Zhou Fan, Iain~M. Johnstone, and Yi~Sun.
\newblock Spiked covariances and principal components analysis in
  high-dimensional random effects models, 2018.
\newblock arXiv 1806.09529.

\bibitem[Kato(1980)]{Kato1980}
T.~Kato.
\newblock \emph{Perturbation {T}heory of {L}inear {O}perators}.
\newblock Springer-Verlag, 1980.

\bibitem[Lee et~al.(2010)Lee, Zou, and Wright]{lzw10}
Seunggeun Lee, Fei Zou, and Fred~A. Wright.
\newblock Convergence and prediction of principal component scores in
  high-dimensional settings.
\newblock \emph{Ann. Statist.}, 38\penalty0 (6):\penalty0 3605--3629, 12 2010.

\bibitem[Mathias(1993)]{math93}
Roy Mathias.
\newblock The {H}adamard operator norm of a circulant and applications.
\newblock \emph{SIAM J. Matrix Anal. Appl.}, 14\penalty0 (4):\penalty0
  1152--1167, 1993.

\bibitem[Morales-Jimenez et~al.(2018)Morales-Jimenez, Johnstone, McKay, and
  Yang]{mjmy18}
David Morales-Jimenez, Iain~M. Johnstone, Matthew~R. McKay, and Jeha Yang.
\newblock Asymptotics of sample correlation matrices for spiked models, 2018.
\newblock arXiv 1810.10214.

\bibitem[Muirhead(1982)]{muir82}
R.~J. Muirhead.
\newblock \emph{Aspects of Multivariate Statistical Theory}.
\newblock Wiley, 1982.

\bibitem[Paul(2005)]{Paul2005}
D.~Paul.
\newblock \emph{Nonparametric Estimation of Principal Components}.
\newblock PhD thesis, Stanford University, 2005.

\bibitem[Paul(2007)]{paul07}
Debashis Paul.
\newblock Asymptotics of sample eigenstructure for a large dimensional spiked
  covariance model.
\newblock \emph{Statistica Sinica}, 17:\penalty0 1617--1642, 2007.

\bibitem[Shi(2013)]{shi13}
Dai Shi.
\newblock Asymptotic joint distribution of extreme sample eigenvalues and
  eigenvectors in the spiked population model, 2013.
\newblock arXiv:1304.6113.

\bibitem[Silverstein and Bai(1995)]{siba95}
Jack~W. Silverstein and Z.~D. Bai.
\newblock On the empirical distribution of eigenvalues of a class of large
  dimensional random matrices.
\newblock \emph{Journal of Multivariate Analysis}, 54:\penalty0 175--192, 1995.

\bibitem[Yang and Johnstone(2018)]{yajo17}
Jeha Yang and Iain~M. Johnstone.
\newblock Edgeworth correction for the largest eigenvalue.
\newblock \emph{Statistica Sinica}, 28:\penalty0 2541--2564, 2018.

\bibitem[Yao et~al.(2015)Yao, Zheng, and Bai]{yzb15}
Jianfeng Yao, Shurong Zheng, and Zhidong Bai.
\newblock \emph{Large sample covariance matrices and high-dimensional data
  analysis}, volume~39 of \emph{Cambridge Series in Statistical and
  Probabilistic Mathematics}.
\newblock Cambridge University Press, New York, 2015.

\end{thebibliography}

\end{document}